\documentclass[11pt]{article}
\usepackage[T1]{fontenc}
\usepackage[utf8]{inputenc}
\usepackage[margin=2cm]{geometry}
\usepackage{amsmath,amsthm,amssymb,mathrsfs}
\usepackage{graphicx}
\usepackage{hyperref}
\hypersetup{colorlinks=true,linkcolor=blue,citecolor=blue,urlcolor=cyan}

\newtheorem{prop}{Proposition}[section]
\newtheorem{cor}[prop]{Corollary}
\newtheorem{lem}[prop]{Lemma}
\newtheorem{theo}[prop]{Theorem}
\theoremstyle{definition}
\newtheorem{examp}[prop]{Example}
\theoremstyle{remark}
\newtheorem{rmk}[prop]{Remark}

\newcounter{hypA}
\newenvironment{hypA}{\refstepcounter{hypA}\begin{itemize}
  \item[({\bf A\arabic{hypA}})]}{\end{itemize}}
\newcounter{hypH}
\newenvironment{hypH}{\refstepcounter{hypH}\begin{itemize}
  \item[({\bf H\arabic{hypH}})]}{\end{itemize}}

\DeclareMathOperator*{\argmin}{arg\,min}
\newcommand{\tr}{\mathrm{Tr}}
\newcommand{\Prob}{\mathrm{Prob}}
\newcommand{\point}{\,\cdot\,}
\newcommand{\cchi}{\protect\raisebox{2pt}{$\chi$}}
\newcommand{\BB}{\mathbb{B}}
\newcommand{\CC}{\mathbb{C}}
\newcommand{\EE}{\mathbb{E}}
\newcommand{\MM}{\mathbb{M}}

\newcommand{\PP}{\mathbb{P}}
\newcommand{\QQ}{\mathbb{Q}}
\newcommand{\RR}{\mathbb{R}}
\newcommand{\UU}{\mathbb{U}}
\newcommand{\VV}{\mathbb{V}}
\newcommand{\XX}{\mathbb{X}}
\newcommand{\YY}{\mathbb{Y}}
\newcommand{\Ca}{{\cal C}}
\newcommand{\Da}{{\cal D}}
\newcommand{\Ha}{{\cal H}}
\newcommand{\Ra}{{\cal R}}
\newcommand{\Sa}{{\cal S}}
\newcommand{\Ua}{{\cal U}}
\newcommand{\Cl}{\underline{\texttt{C}}}
\newcommand{\Cu}{\overline{\texttt{C}}}

\begin{document}

\begin{center}
{\Large \textbf{On Bridging Mixture Distributions}}

\vspace{0.5cm}

PIERRE DEL MORAL$^{1}$, AJAY JASRA$^{2}$ \& KE ZHAO$^{3}$

{\footnotesize $^{1}$Centre de Recherche Inria Bordeaux Sud-Ouest, Talence,  FR.}\\
{\footnotesize $^{2}$Computing \& Mathematical Sciences Division,  Mohamed Bin Zayed University of Artificial Intelligence,  Abu Dhabi, UAE.}\\
{\footnotesize $^{3}$School of Data Science,  The Chinese University of Hong Kong, Shenzhen,  Shenzhen,  CN.}\\
{\footnotesize E-Mail:\,}
\texttt{\emph{\footnotesize pierre.del-moral@inria.fr}};
\texttt{\emph{\footnotesize Ajay.Jasra@mbzuai.ac.ae}};
\texttt{\emph{\footnotesize kezhao@link.cuhk.edu.cn}}

\begin{abstract}
The construction and the stability of bridges between mixture probability measures are investigated. Given a collection of Markov transitions bridging the components of two mixtures, together with a coupling of the mixture labels, we construct a bridge between the mixtures on an extended state space, and we show that this bridge inherits the entropic optimality of the component bridges. We then turn to Gaussian mixtures bridged by Gaussian Schr\"odinger bridges, whose parameters are generally unknown and must be estimated. We prove a $2$-Wasserstein continuity theorem between the exact bridge and its plug-in approximation. The analysis rests on a representation of the Riccati fixed point map in inverse coordinates; this map is globally $1$-Lipschitz and requires no matrix inversion. These results are applied to three estimation schemes, with $d$ the dimension and $N$ the number of samples. For Gaussian mixtures estimated by the Expectation-Maximization algorithm, the squared $2$-Wasserstein error is of order $(d\log N/N)^{1/2}+d^2\log N/N$ with probability at least $1-10N^{-1}$. For single Gaussian marginals estimated by their sample moments, it is of order $d^2/N$ in expectation. For the regularized empirical Monge map based on an $\epsilon$-inflation of the sample covariance, the mean squared error is of order $d^2\{(1+\epsilon^{-1})/N+\epsilon^2\}$; as soon as $N$ is a sufficiently large multiple of $d$, it reduces to $d^2/N+d\,\epsilon^2$, up to an exponentially small term. These estimates are illustrated numerically.

\end{abstract}
\end{center}

\section{Introduction}

\subsection{Literature review}

The problem of bridging between two probability measures, that is to sample from one probability and then to move that sample to the other probability, arises in a variety of areas of applied mathematics, including probability theory, statistics and generative artificial intelligence (AI); the reader is referred to \cite{csis,cut,bort,sinkhorn_rev,genev,leonard} and the references therein. There are by now multiple methodologies to bridge between probability measures, often based on the Schr\"odinger bridge, and perhaps the most popular approximations of these are Sinkhorn algorithms/iterative proportional fitting and related methods; see for instance \cite{cut,bort,sinkhorn_rev,gold,koro2,koro1}. Whilst there exist several mathematical results for the convergence of these approaches, they often rely on machine learning methods, such as deep neural networks, which themselves are not fully understood from a mathematical perspective, or make rather strong assumptions, see e.g.~\cite{maeda,pool,stromme}, or have approximations which cannot be sampled easily \cite{har}. Some rigorous works include \cite{gold,kato}, but they do not always correspond to numerically implementable methods.

The underlying principle of this article is to design a bridge methodology which can be analyzed mathematically, with explicit convergence rates, albeit relying itself on estimation methodology. We focus on the problem of bridging between mixture probability measures, for instance finite mixture models (see e.g.~\cite{mix}) or continuous mixtures. These probability measures are rather rich in terms of their approximation properties (see e.g.~\cite{mix_approx}). For instance, given two probability measures and empirical samples of each, in many situations one can approximate the two probability measures by convolving the empirical distributions with a Gaussian kernel, leading to exactly the problem of bridging between mixtures, up to some bias induced by the approximation.

Bridging between mixtures has been considered by several authors. The article \cite{chen} considers discrete Gaussian mixtures; since the squared Wasserstein distance between Gaussians is known, one can focus on computing some optimal jump probabilities between components and ultimately a geodesic between Gaussian mixtures. That article does not consider parameter uncertainty, which is one of the core contributions of the present work. The article \cite{rapak} (see also \cite{alberg} and \cite{rapak1}) considers Gaussian mixtures and provides a sub-optimal solution to the Schr\"odinger bridge problem based upon mixtures of Gaussian Schr\"odinger bridges. We consider a slightly more general problem and we provide error bounds, whereas \cite{rapak} does not consider this latter issue from a mathematical perspective. The work \cite{rapak} also focuses on low to moderate dimensional problems; thanks to the rationalization of the Riccati map and to the covariance inflation of the Monge map, our approach could potentially tackle higher dimensional problems, although this is not the focus of this work. As in \cite{rapak}, and as in \cite{koro2,koro1}, we seek a `fast' methodology for bridging between mixtures; the latter approaches focus on a more general problem and rely on several approximations, which can make the methodology challenging to analyze mathematically.

\subsection{Description of the models}

Consider two mixture probability measures $p$ and $q$, with mixing distributions $\imath$ and $\jmath$ on some label spaces $\UU$ and $\VV$, and components $P(u,\point)$ and $Q(v,\point)$ indexed by the labels $u\in\UU$ and $v\in\VV$; see (\ref{def-pq}). The idea is to transport the labels and the components separately. Given a Markov transition $J$ moving the labels from $\imath$ to $\jmath$, and for any pair of labels $(u,v)$ a Markov transition $\BB_{u,v}$ moving the component $P(u,\point)$ onto the component $Q(v,\point)$, we consider the product Markov transition $\MM$ defined in (\ref{def-MM}), acting on labelled states. This transition maps the joint distribution $\PP$ of a label and a state onto its target counterpart $\QQ$. Integrating out the labels, we obtain a Markov transition $M$ bridging $p$ to $q$; see Section~\ref{sec:product}.

The second part of the article is concerned with Gaussian mixtures (\ref{def-gauss-comp}), where $\BB_{u,v}$ is the Gaussian Schr\"odinger bridge between the components, w.r.t.~a linear Gaussian reference transition. In this context, $\BB_{u,v}(x,\point)$ is a Gaussian distribution with explicit mean and covariance matrix, expressed in terms of the fixed point of a Riccati map; see (\ref{def-Y}). In practice the parameters of the components are unknown and replaced by some estimates; we denote by $(\widehat{\PP},\widehat{\MM})$ the corresponding plug-in approximations of $(\PP,\MM)$.

\subsection{Statement of some main results}

The first result of the article, Proposition~\ref{prop:var}, shows that whenever the component bridges $\BB_{u,v}$ are Schr\"odinger bridges w.r.t.~some reference transition, the bridge $\PP\times\MM$ is entropically optimal on the extended state space, for the given coupling $\pi$ of the labels defined in (\ref{pi-def}).

Our second objective is to quantify the error induced by the estimation of the parameters. Theorem~\ref{thm:conti} provides the $2$-Wasserstein continuity estimate (\ref{conti-bound}). The error of the plug-in bridge is controlled by the errors on the components, inflated by a Lipschitz constant $a$ of $\MM$ w.r.t.~the states, and by the averaged errors on the component bridges; an extension to estimated weights is provided in Corollary~\ref{cor:unknown}. The second term in (\ref{conti-bound}) depends on the Riccati fixed points. The key observation is that the Riccati map can be expressed in inverse coordinates as the matrix function $F$ defined in (\ref{def:h}). This function is defined on the whole set of positive semi-definite matrices and it is globally $1$-Lipschitz w.r.t.~any unitarily invariant norm (Proposition~\ref{prop:ricc-alt} and Lemma~\ref{lem:h-lipschitz}). Under some uniform conditioning assumptions on the covariance matrices, this property yields the stability estimate (\ref{main-bound}) stated in Theorem~\ref{theo-main}: the squared $2$-Wasserstein error is bounded by the dimension $d$ times a linear combination of averaged estimation errors on the means and the covariance matrices, weighted by some moments of the estimated covariance matrices.

These results are applied to several estimation schemes in Section~\ref{sec:empirical}. For Gaussian mixtures estimated by the Expectation-Maximization (EM) algorithm, Corollary~\ref{cor:em} shows that, with probability at least $1-10N^{-1}$, the squared $2$-Wasserstein error is of order $(d\log N/N)^{1/2}+d^2\log N/N$, where $N$ stands for the number of samples; see (\ref{ineq:mix-bridge}). For a single pair of Gaussian marginals estimated by their sample moments, Corollary~\ref{cor-gauss-rate} shows that the expected squared $2$-Wasserstein distance between the exact and the empirical Gaussian Schr\"odinger bridges is of order $d^2/N$. In the vanishing regularization limit, the Gaussian Schr\"odinger bridge converges to the Monge map. For the regularized empirical Monge map associated with an $\epsilon$-inflation of the sample covariance, Theorem~\ref{theo-Monge} shows that the mean squared error is of order $d^2\{(1+\epsilon^{-1})/N+\epsilon^2\}$, and that the factor $\epsilon^{-1}$ only multiplies an exponentially small term as soon as $N$ is a sufficiently large multiple of $d$; see (\ref{ineq:Monge2}) and (\ref{ineq:Monge3}). The numerical experiments presented in Section~\ref{sec:num} illustrate these rates, including the case of Gaussian mixtures estimated by the EM algorithm from unlabeled samples.

\subsection{Organisation of the article}

The article is structured as follows: In Section~\ref{sec:pre} we introduce the notation used throughout the article and we recall the Bures--Wasserstein geometry of Gaussian measures. Section~\ref{sec:product} introduces bridge products and their variational characterization. Section~\ref{sec:gauss-mix} is dedicated to Gaussian mixture bridges and to the continuity theorem. The perturbation analysis of Riccati maps and the main stability theorem are developed in Section~\ref{sec:stability}. Section~\ref{sec:empirical} is concerned with finite-sample convergence rates and with the Monge limit. Section~\ref{sec:num} presents some numerical illustrations. The appendix houses the proofs of some technical results.

\section{Preliminaries}\label{sec:pre}

\subsection{Notation}

The $2$-Wasserstein distance between probability measures $\eta_1$ and $\eta_2$ on some metric space $(\XX,\rho_{\XX})$ is given by
$$
\Da_{2}(\eta_1,\eta_2):= \inf\left(\int\rho_{\XX} (x_1,x_2)^2~\pi(d(x_1,x_2))\right)^{1/{2}},
$$
where the infimum is taken over the convex set $\Ca(\eta_1,\eta_2)$ of probability measures $\pi(d(x_1,x_2))$ with marginal $\eta_1$ w.r.t.~the first coordinate and marginal $\eta_2$ w.r.t.~the second coordinate.

Let $\Sa^0_d$ be the set of positive semi-definite matrices in $\RR^{d \times d}$, and $\Sa^+_d\subset \Sa^0_d$ the subset of positive definite matrices. We denote by $\nu_{m,\sigma}$ the Gaussian distribution on $\RR^d$ with mean $m\in\RR^d$ and covariance matrix $\sigma \in \Sa^+_d$. The geometric mean $\sigma_1\sharp\sigma_2$ of two positive definite matrices $\sigma_1,\sigma_2\in\Sa^+_d$ is defined by
\begin{equation}\label{sym-sharp}
\sigma_1~\sharp~\sigma_2=\sigma_2~\sharp~\sigma_1:=\sigma_2^{1/2}~ \left(\sigma_2^{-1/2}~\sigma_1~\sigma_2^{-1/2}\right)^{1/2}~\sigma_2^{1/2}.
\end{equation}
We recall that the geometric mean is the unique positive definite solution of the Riccati equation
$$
(\sigma_1~\sharp~\sigma_2)~\sigma_1^{-1}~(\sigma_1~\sharp~\sigma_2)=\sigma_2 .
$$

We denote by $\lambda_{\min}(w)$ and $\lambda_{\max}(w)$ the minimal and the maximal eigenvalues of a symmetric matrix $w\in\RR^{d\times d}$. The Frobenius and the spectral norms of a matrix $w$ are defined by
$$
\Vert w\Vert_F^2=\tr(w^{\prime}w)
\quad\mbox{\rm and}\quad
\Vert w\Vert_2^2=\lambda_{\max}(w^{\prime}w),
$$
with the trace operator $\tr(\point)$ and the transpose $w^{\prime}$ of $w$. We recall the equivalence property
$$
\Vert w\Vert_{2}\leq \Vert w\Vert_{F}\leq \sqrt{d}~\Vert w\Vert_{2}.
$$
For any $\sigma\in\Sa_{d}^+$ we also have
$$
\Vert \sigma^{1/2}\Vert_2^2=\Vert \sigma\Vert_2,
\qquad
\Vert \sigma^{-1}\Vert_2=\frac{1}{\lambda_{\min}(\sigma)}
\quad\mbox{\rm and}\quad
\Vert \sigma^{1/2}\Vert_F^2=\tr(\sigma).
$$
For any vector $x \in \RR^d$, we write $\Vert x\Vert_2$ for the Euclidean norm and $\Vert x\Vert_1=\sum_{i=1}^d|x_i|$ for the $l_1$ norm. Given $\sigma \in \Sa_d^+$, the Mahalanobis norm is defined by $\Vert x\Vert_{\sigma}:=\Vert\sigma^{-1/2}x\Vert_2$; in this context, we have
\begin{equation}\label{Mnorm}
\Vert x\Vert_{\sigma}^2=x'\sigma^{-1}x\ge \lambda_{\max}(\sigma)^{-1}~\Vert x\Vert_2^2 .
\end{equation}
Finally, the relative entropy of a probability $Q$ w.r.t.~a probability $P$ on the same state space, with $Q\ll P$, is defined by
$$
\Ha(Q|P):=\int\log\left(\frac{dQ}{dP}\right)dQ .
$$

We also set $a\wedge b=\min(a,b)$ and $a\vee b=\max(a,b)$, for any $a,b\in\RR$. With a slight abuse of notation, we denote by $I$ the identity matrix, as well as the identity integral operator. Unless otherwise stated, we write $C$ for some finite constants whose values may vary from line to line, and we write $C_k$ when their values may depend on some parameter $k$.

\subsection{Preliminary estimates}

We end this section with some rather well known estimates that will be used throughout the article. For any $\sigma_1,\sigma_2\in\Sa_{d}^+$ we have the Ando--van Hemmen inequality
\begin{equation}\label{square-root-key-estimate}
\Vert \sigma_1^{1/2}-\sigma_2^{1/2}\Vert \leq \left[\lambda^{1/2}_{\rm min}(\sigma_1)+\lambda^{1/2}_{\rm min}(\sigma_2)\right]^{-1}~\Vert \sigma_1-\sigma_2\Vert,
\end{equation}
which holds for any unitarily invariant matrix norm $\Vert\point\Vert$, including the spectral and the Frobenius norms; see for instance \cite[Proposition 3.2]{hemmen}. Using (\ref{sym-sharp}) we also find the identities
$$
\sigma_1^{-1}~\sharp~\sigma_2=\sigma_1^{-1/2}~ \left(\sigma_1^{1/2}~\sigma_2~\sigma_1^{1/2}\right)^{1/2}~\sigma_1^{-1/2}
\Longrightarrow
(\sigma_1^{-1}~\sharp~ \sigma_2)~\sigma_1~(\sigma_1^{-1}~\sharp~ \sigma_2)=\sigma_2 .
$$
For any given $m_1,m_2\in\RR^d$ and $\sigma_1,\sigma_2 \in \Sa^+_d$, we have
\begin{equation}\label{burg-def}
\Da_2(\nu_{m_1,\sigma_1},\nu_{m_2,\sigma_2})^2=D_{bw}(\sigma_1,\sigma_2)^2+\Vert m_1-m_2\Vert_2^2,
\end{equation}
with the Bures--Wasserstein distance \cite{bathia-2} on $\Sa^+_d$ given by
$$
D_{bw}(\sigma_1,\sigma_2)^2=\tr(\sigma_1)+\tr(\sigma_2)-2\tr\left(\left( \sigma_2^{1/2}~\sigma_1~\sigma_2^{1/2}\right)^{1/2}\right)
=\min_{w\in \Ua(d) }\Vert \sigma_1^{1/2}-\sigma_2^{1/2}w\Vert_F^2,
$$
where $\Ua(d)$ stands for the group of unitary matrices. Choosing $w=I$ and using (\ref{square-root-key-estimate}) we check that
\begin{equation}\label{eqD2-g}
\Da_2(\nu_{m_1,\sigma_1},\nu_{m_2,\sigma_2})^2\leq \left[\lambda^{1/2}_{\rm min}(\sigma_1)+\lambda^{1/2}_{\rm min}(\sigma_2)\right]^{-2}~\Vert \sigma_1-\sigma_2\Vert_F^2+\Vert m_1-m_2\Vert_2^2.
\end{equation}
We also recall that the optimal coupling between two Gaussian distributions is deterministic. More precisely, the Monge map
\begin{equation}\label{def-Monge}
T(x):=m_2+(\sigma_1^{-1}~\sharp~ \sigma_2)~(x-m_1)
\end{equation}
pushes $\nu_{m_1,\sigma_1}$ forward to $\nu_{m_2,\sigma_2}$, and we have
\begin{equation}\label{eq:Monge}
\Da_2(\nu_{m_1,\sigma_1},\nu_{m_2,\sigma_2})^2=\int \nu_{m_1,\sigma_1}(dx)~\Vert x-T(x)\Vert_2^2 .
\end{equation}

Consider some metric spaces $(\XX,\rho_{\XX})$ and $(\YY,\rho_{\YY})$ and equip $\XX\times \YY$ with the distance
$$
\rho((x,y),(\overline{x},\overline{y}))^2=\rho_{\XX}(x,\overline{x})^2+\rho_{\YY}(y,\overline{y})^2.
$$
Given a probability measure $\mu$ on $\XX$ and a Markov kernel $K$ from $\XX$ into $\YY$, we denote by $\mu\times K$ the joint distribution of a pair of random variables $(X,Y)$, where $X$ is distributed according to $\mu$ and, given $X=x$, the random variable $Y$ is distributed according to $K(x,\point)$; that is
$$
(\mu\times K)(d(x,y)):=\mu(dx)~K(x,dy).
$$
Next we present an estimate of the $2$-Wasserstein distance between two probability measures of this form. Roughly speaking, the distance between $\mu\times K$ and $\overline{\mu}\times \overline{K}$ is controlled by the distance between the first marginals, amplified by the Lipschitz constant of $\overline{K}$, and by the distance between the kernels, averaged w.r.t.~$\mu$. Here and below we implicitly assume that the optimal couplings between $K(x,\point)$ and $\overline{K}(x,\point)$ (resp.~between $\overline{K}(x,\point)$ and $\overline{K}(\overline{x},\point)$) can be selected measurably in $x$ (resp.~in $(x,\overline{x})$); this is always possible on Polish spaces (see for instance \cite[Corollary 5.22]{villani}), and in the Gaussian situations considered in this article these couplings are explicit. The proof of the next technical lemma is rather standard; for the convenience of the reader it is provided in the appendix, see Section~\ref{sec:lem_prf}.

\begin{lem}\label{Lem-Tex}
Let $\gamma\in\Ca(\mu,\overline{\mu})$ be a coupling of $\mu$ and $\overline{\mu}$, and assume there exists $a\geq 0$ such that for $\gamma$-almost every $(x,\overline{x})\in\XX^2$ we have
\begin{equation}\label{est-hyp}
\Da_2\left( \overline{K}(x,\point), \overline{K}(\overline{x},\point)\right)\leq a~\rho_{\XX}(x,\overline{x}).
\end{equation}
Then we have that
\begin{equation}\label{est-lem-tex}
\begin{split}
\Da_2\left(\mu\times K,\overline{\mu}\times \overline{K}\right)\leq{}&
\sqrt{1+a^2}~\left(\int \gamma(d(x,\overline{x}))~\rho_{\XX}(x,\overline{x})^2\right)^{1/2}\\
&+\left( \int\mu(dx)~ \Da_2\left(K(x,\point),\overline{K}(x,\point)\right)^2 \right)^{1/2}.
\end{split}
\end{equation}
In particular, if (\ref{est-hyp}) holds for any $x,\overline{x}\in\XX$, choosing for $\gamma$ an optimal coupling we find that
$$
\Da_2\left(\mu\times K,\overline{\mu}\times \overline{K}\right)\leq 
\sqrt{1+a^2}~\Da_2(\mu,\overline{\mu})+\left( \int\mu(dx)~ \Da_2\left(K(x,\point),\overline{K}(x,\point)\right)^2 \right)^{1/2}.
$$
\end{lem}

\section{Bridging Mixtures}\label{sec:product}

In this section we consider bridging between two mixture target distributions. The basic idea is to consider a bridge between two `components' of the mixtures and to bridge on a product space (bridge products, Section~\ref{sec:bridge_prod}); we then show that this can be used on the marginals, that is on the mixtures (Section~\ref{sec:marg}). Finally, we consider an associated variational formulation (Section~\ref{sec:vari_form}), which provides a type of optimality of the proposed procedure.

\subsection{Bridge products}\label{sec:bridge_prod}

Consider probability measures $\imath(du)$ and $\jmath(dv)$ on some state spaces $\UU$ and $\VV$. These spaces will correspond to the labels of the mixture components, for instance $\UU=\{1,\dots,k_U\}$ and $\VV=\{1,\dots,k_V\}$, but the formulation does not constrain one to such a finite mixture scenario. In addition, consider Markov transitions $P(u,dx)$ and $Q(v,dy)$ from $\UU$ and $\VV$ into some state spaces $\XX$ and $\YY$. In this notation, the label $u$ indexes the component $P(u,\point)$ of the source mixture, and the label $v$ indexes the component $Q(v,\point)$ of the target mixture. We introduce the probability measures $\PP$ and $\QQ$ on the product spaces $\UU\times\XX$ and $\VV\times\YY$ defined by
$$
\PP(d(u,x)):=\imath(du)~P(u,dx)
\quad\mbox{\rm and}\quad
\QQ(d(v,y)):=\jmath(dv)~Q(v,dy).
$$
In other words, $\PP$ is the joint distribution of a label $u$ sampled from $\imath$ and of a state $x$ sampled from the corresponding component $P(u,\point)$; the same interpretation holds for $\QQ$. By construction, the marginals of $\PP$ and $\QQ$ on $\XX$ and $\YY$ are the two mixtures.

To transport $\PP$ to $\QQ$ we use two ingredients: a Markov transition $J(u,dv)$ from $\UU$ into $\VV$, which moves the labels, and, for any pair of labels $(u,v)\in \UU\times\VV$, a Markov transition $\BB_{u,v}(x,dy)$ from $\XX$ into $\YY$, which moves the component $u$ onto the component $v$. More precisely, we assume that
\begin{equation}\label{eq:bridge_prop_B}
\int_{x\in \XX}P(u,dx)~\BB_{u,v}(x,dy)=Q(v,dy)
\quad\mbox{\rm and}\quad (\imath J)(dv):=\int_{u\in \UU} \imath(du)~J(u,dv)=\jmath(dv).
\end{equation}
In other words, the probability distribution $\pi$ on $\UU\times\VV$ defined by
\begin{equation}\label{pi-def}
\pi(d(u,v)):=\imath(du)~J(u,dv)
\end{equation}
is a bridge from $\imath(du)$ to $\jmath(dv)$. In addition, for any $(u,v)\in\UU\times\VV$, the probability distribution $\Pi((u,v),\point)$ on $\XX\times\YY$ defined by
$$
\Pi((u,v),d(x,y)):=P(u,dx)~ \BB_{u,v}(x,dy)
$$
is a bridge from $P(u,dx)$ to $Q(v,dy)$. We associate with these objects the Markov transition $\MM$ from $\UU\times\XX$ into $\VV\times\YY$ defined by the product transition formula
\begin{equation}\label{def-MM}
\MM((u,x),d(v,y)):=J(u,dv)~\BB_{u,v}(x,dy).
\end{equation}
Starting from a labelled state $(u,x)$, the transition $\MM$ first samples a new label $v$ with distribution $J(u,\point)$, and then a new state $y$ with distribution $\BB_{u,v}(x,\point)$. By construction, we have
$$
\PP(d(u,x))~\MM((u,x),d(v,y))=\imath(du)~J(u,dv)~P(u,dx)~\BB_{u,v}(x,dy)
\Longrightarrow \PP\MM=\QQ,
$$
that is, the probability distribution $\PP\times\MM$ is a bridge from $\PP$ to $\QQ$. We also have the bridge product formula
$$
\PP(d(u,x))~\MM((u,x),d(v,y))=\pi(d(u,v))~\Pi((u,v),d(x,y)),
$$
which states that the bridge $\PP\times\MM$ can be sampled by first sampling the labels $(u,v)$ with the label bridge $\pi$, and then the states $(x,y)$ with the component bridge $\Pi((u,v),\point)$.

\subsection{Marginal bridges}\label{sec:marg}

The marginal distributions $p$ and $q$ of $\PP$ and $\QQ$ w.r.t.~the second coordinates (i.e.~on $\XX$ and $\YY$ respectively) are defined by
\begin{equation}\label{def-pq}
p(dx)=(\imath P)(dx):=\int_{u\in \UU}\imath(du)~P(u,dx)
\quad\mbox{\rm and}\quad
q(dy)=(\jmath Q)(dy):=\int_{v\in \VV}\jmath (dv)~Q(v,dy).
\end{equation}
The probability measures $p$ and $q$ are the mixture targets: for example if $\XX=\YY=\RR^d$, $\UU=\{1,\dots,k_U\}$ and $\VV=\{1,\dots,k_V\}$, then $p$ and $q$ could be finite mixtures of Gaussian probability measures.

We assume that for any $(u,v)\in \UU\times\VV$ we have $P(u,\point)\ll p$ and $Q(v,\point)\ll q$. In this situation, we have the Bayes' formulae
$$
\PP(d(u,x))=p(dx)~P_{\imath}(x,du)
\quad\mbox{\rm and}\quad
\QQ(d(v,y))=q(dy)~Q_{\jmath}(y,dv),
$$
with the conditional distributions
$$
P_{\imath}(x,du):=\frac{\imath(du)~P(u,dx)}{\int \imath(dw)~P(w,dx)}
\quad\mbox{\rm and}\quad
Q_{\jmath}(y,dv):=\frac{\jmath(dv)~Q(v,dy)}{\int\jmath(dw)~Q(w,dy)}.
$$
In standard probabilistic terms, $P_{\imath}(x,du)$ is the conditional distribution of the mixture label $u$ given the state $x$; a similar interpretation holds for $Q_{\jmath}(y,dv)$.

The bridge between the mixtures acts on the states only: given a state $x$, it samples a label $u$ from the conditional distribution $P_{\imath}(x,\point)$, then moves the labelled state $(u,x)$ with the transition $\MM$, and finally forgets the labels. More precisely, we set
$$
M(x,dy):=\int_{(u,v)\in\UU\times\VV}P_{\imath}(x,du)~\MM((u,x),d(v,y)),
$$
and we note that
$$
p(dx)~M(x,dy)=\int_{(u,v)\in\UU\times\VV}\PP(d(u,x))~\MM((u,x),d(v,y)).
$$
This implies that
\begin{eqnarray*}
(p M)(dy)&:=&\int_{x\in\XX}p(dx)~M(x,dy)=\int_{v\in\VV}\int_{(u,x)\in \UU\times \XX}\PP(d(u,x))~\MM((u,x),d(v,y))\\
&=&\int_{v\in \VV} \QQ(d(v,y))=q(dy),
\end{eqnarray*}
from which we check that $pM=q$. In other words, if $x$ is sampled from $p$ and then moved with $M$, the resulting probability measure on $\YY$ is exactly $q$. This is an example of a bridge between mixtures; from a practical perspective, it requires (at least) $p$ and $q$ to be known and one to be able to sample from $P_{\imath}$.

\subsection{Variational formulation}\label{sec:vari_form}

We now show that the construction above inherits the optimality of the component bridges. The optimality is measured w.r.t.~a reference dynamics: consider a Markov transition $K(x,dy)$ from $\XX$ into $\YY$ (in Section~\ref{sec:gauss-mix} a linear Gaussian transition), started from the component $u$, that is
$$
R((u,v),d(x,y)):=P(u,dx)~K(x,dy).
$$
By the chain rule for the relative entropy, we have
$$
\Ha\left(\Pi((u,v),\point)~|~R((u,v),\point)\right)=\int P(u,dx)~\Ha\left(\BB_{u,v}(x,\point)~|~K(x,\point)\right).
$$
We further assume that each component bridge is the Schr\"odinger bridge w.r.t.~the reference $R$; that is, for $\pi$-almost every $(u,v)\in \UU\times\VV$ we have
$$
\Pi((u,v),\point)=\argmin\Ha\left(\Gamma((u,v),\point)~|~R((u,v),\point)\right),
$$
where the minimum is taken over all coupling Markov transitions $\Gamma((u,v),d(x,y))$ such that
\begin{equation}\label{ent-1}
\int_{y\in\YY}\Gamma((u,v),d(x,y))=P(u,dx)
\quad\mbox{\rm and}\quad
\int_{x\in\XX}\Gamma((u,v),d(x,y))=Q(v,dy).
\end{equation}
In this context we have the following variational formula.

\begin{prop}\label{prop:var}
We have that
$$
\Ha\left(\pi\times \Pi~|~\pi\times R\right)=\inf_{\Gamma}\Ha\left(\pi\times \Gamma~|~\pi\times R\right),
$$
where the infimum is taken over all coupling Markov transitions $\Gamma$ such that (\ref{ent-1}) holds for $\pi$-almost every $(u,v)\in \UU\times\VV$.
\end{prop}
\begin{proof}
For any coupling Markov transition $\Gamma$ satisfying (\ref{ent-1}), we have
$$
\Ha\left(\pi\times \Gamma~|~\pi\times R\right)=\int_{(u,v)\in \UU\times\VV}\pi(d(u,v))~\Ha\left(\Gamma((u,v),\point)~|~R((u,v),\point)\right).
$$
The marginal constraints (\ref{ent-1}) are pointwise in $(u,v)$, so that the integral is minimized by minimizing each integrand separately. By assumption, each integrand is minimized by $\Pi((u,v),\point)$ for $\pi$-almost every $(u,v)$, and $\Pi$ satisfies (\ref{ent-1}) by (\ref{eq:bridge_prop_B}). This ends the proof of the proposition.
\end{proof}

\section{Gaussian Mixture Bridges}\label{sec:gauss-mix}

We now consider the ideas of the previous section in the context of Gaussian mixtures. As we have remarked above, the construction is only of use if $p$, $q$ and so on are available. In practice, especially in generative AI, these probabilities must be estimated in some way, and so must the bridges $\BB_{u,v}$. To that end, the main result of this section is a continuity result (Theorem~\ref{thm:conti}), which provides a bound on the $2$-Wasserstein distance between $\PP\times\MM$ and its estimated counterpart $\widehat{\PP}\times\widehat{\MM}$. This upper-bound is expressed in terms of an averaged $2$-Wasserstein distance between the true and estimated Gaussian components, and an averaged $2$-Wasserstein distance between the true and estimated Gaussian bridges. In a first step we suppose that $\imath(du)$ and $\jmath(dv)$ are given probability measures; the case where these probabilities are estimated is discussed in Section~\ref{sec:unknown_weights}.

\subsection{Gaussian Schr\"odinger bridges}

This formulation follows the exposition of \cite{adm-24} and \cite[Section 5.4]{sinkhorn_rev}; we only present the elements that are relevant for this article. Consider a linear Gaussian reference Markov transition $K_{\theta}(x,dy)$ from $\XX=\RR^d$ into $\YY=\RR^d$, indexed by some parameter
$$
\theta=(\alpha,\beta,\tau)\in \boldsymbol{\Theta}:=\RR^{d}\times\RR^{d\times d}\times\Sa_d^+,
$$
with $\beta$ invertible. That is, for $G$ a $d$-dimensional Gaussian random variable with zero mean and identity covariance, the transition $K_{\theta}(x,\point)$ is the distribution of $\alpha + \beta x + \tau^{1/2}G$. Consider the Gaussian components
\begin{equation}\label{def-gauss-comp}
P(u,dx)=\nu_{m_u,\sigma_u}(dx)
\quad\mbox{\rm and}\quad
Q(v,dy)=\nu_{n_v,\varsigma_v}(dy)
\end{equation}
associated with some maps
$$
u\in \UU\mapsto (m_u,\sigma_u)\in \RR^d\times \Sa_d^+
\quad\mbox{\rm and}\quad
v\in \VV\mapsto (n_v,\varsigma_v)\in \RR^d\times \Sa_d^+.
$$
In this context, the reference transition of Section~\ref{sec:vari_form} takes the form
$$
R_{\theta}((u,v),d(x,y)):=\nu_{m_u,\sigma_u}(dx)~K_{\theta}(x,dy).
$$
The Schr\"odinger bridge between two Gaussian components is again a linear Gaussian transition \cite{adm-24}. More precisely, for any $(u,v)\in\UU\times\VV$, we let $\BB_{u,v}(x,dy)$ be the distribution of the random map
\begin{equation}\label{def-Y}
Y_{u,v}(x):=n_v+\kappa_{u,v}~ \cchi~(x-m_u)+\kappa_{u,v}^{1/2}~G
\quad \mbox{\rm with}\quad
\cchi:=\tau^{-1}\beta .
\end{equation}
The state is moved linearly and then blurred by a Gaussian noise with covariance matrix $\kappa_{u,v}$. In the above display, $\kappa_{u,v}$ stands for the rescaled matrix
$$
\kappa_{u,v}:=\varsigma_v^{1/2}~r_{u,v}~\varsigma_v^{1/2},
$$
where $r_{u,v}$ is the positive definite Riccati fixed point
$$
r_{u,v}=-\frac{\varpi_{u,v}}{2}+\left(\varpi_{u,v}+\frac{\varpi_{u,v}^2}{4}\right)^{1/2}
=\left(I+(\varpi_{u,v}+r_{u,v})^{-1}\right)^{-1}
\quad\mbox{\rm with}\quad
\varpi_{u,v}^{-1}:=\varsigma_v^{1/2}~(\cchi~\sigma_u~  \cchi^{\prime})~\varsigma_v^{1/2}.
$$
Note that $0\leq r_{u,v}\leq I$, so that $\kappa_{u,v}\leq\varsigma_v$. To check this claim, observe that $\varpi_{u,v}+r_{u,v}$ is positive definite, so that the r.h.s.~fixed point equation in the above display yields $r_{u,v}\leq I$. In this context, one also has
$$
\nu_{m_u,\sigma_u}\BB_{u,v}=\nu_{n_v,\varsigma_v}
\Longleftrightarrow
(\kappa_{u,v}~ \cchi)~\sigma_u~(\kappa_{u,v}~ \cchi)^{\prime}+\kappa_{u,v}=\varsigma_v
\Longleftrightarrow
r_{u,v}~\varpi_{u,v}^{-1}~r_{u,v}+ r_{u,v}=I .
$$
The r.h.s.~identity is satisfied by the Riccati fixed point, so that the l.h.s.~equation in (\ref{eq:bridge_prop_B}) holds. For any $(u,v)\in \UU\times\VV$, we recall that the Schr\"odinger Gaussian bridges are characterized by
$$
\nu_{m_u,\sigma_u}\times\BB_{u,v}=\argmin\Ha\left(\Gamma((u,v),\point)~|~R_{\theta}((u,v),\point)\right),
$$
where the minimum is taken over all couplings $\Gamma((u,v),\point)$ of $\nu_{m_u,\sigma_u}$ and $\nu_{n_v,\varsigma_v}$. This result is surveyed in \cite[Theorem 5.8]{sinkhorn_rev} and was originally proved in \cite{adm-24}. In this situation, the mixtures $p$ and $q$ defined in (\ref{def-pq}) are Gaussian mixtures.

\subsection{Regularity conditions}

The $2$-Wasserstein distances considered below are taken between distributions of labelled states. Let $(\UU,\rho_{\UU})$ and $(\VV,\rho_{\VV})$ be metric spaces. We equip the sets of labelled states $\XX_d:=\UU\times\RR^d$ and $\YY_d:=\VV\times\RR^d$ with the metrics defined, for any $x_d:=(u,x)$, $\overline{x}_d:=(\overline{u},\overline{x})\in\XX_d$ and $y_d:=(v,y)$, $\overline{y}_d:=(\overline{v},\overline{y})\in\YY_d$, by
$$
\rho_{\XX_d}(x_d,\overline{x}_d)^2:=\rho_{\UU}\left(u,\overline{u}\right)^2+\Vert x-\overline{x}\Vert_2^2
\quad\mbox{\rm and}\quad
\rho_{\YY_d}(y_d,\overline{y}_d)^2:=\rho_{\VV}\left(v,\overline{v}\right)^2+\Vert y-\overline{y}\Vert_2^2 .
$$
We also let $\rho$ be the metric on $\XX_d\times \YY_d$ defined by
$$
\rho\left((x_d,y_d),(\overline{x}_d,\overline{y}_d)\right)^2=\rho_{\XX_d}(x_d,\overline{x}_d)^2+\rho_{\YY_d}(y_d,\overline{y}_d)^2.
$$
We consider two local continuity conditions on the transition $\MM$. The first one only compares states carrying the same label $u$; it is the only one needed when the mixture weights are known (Theorem~\ref{thm:conti}).

\begin{hypH}\label{ass:cont_cond_x}
There exists $a\geq 0$ such that for $\imath$-almost every $u\in\UU$ and any $x,\overline{x}\in\RR^d$ we have that
\begin{equation}\label{continuity-cond-x}
\Da_2\left(\MM((u,x),\point),\MM((u,\overline{x}),\point)\right)\leq a~\Vert x-\overline{x}\Vert_2.
\end{equation}
\end{hypH}
The second one is a global version on $\XX_d$; it is only used when the mixture weights are estimated (Section~\ref{sec:unknown_weights}).

\begin{hypH}\label{ass:cont_cond}
There exists $a\geq 0$ such that for any $x_d,\overline{x}_d\in\XX_d$ we have that
\begin{equation}\label{continuity-cond}
\Da_2\left(\MM(x_d,\point),\MM(\overline{x}_d,\point)\right)\leq a~\rho_{\XX_d}(x_d,\overline{x}_d).
\end{equation}
\end{hypH}

\begin{examp}\label{ex}
Assume we have a $2$-Wasserstein optimal coupling $C((u,\overline{u}),d(v,\overline{v}))$ with marginals $J(u,dv)$ and $J(\overline{u},d\overline{v})$. For instance, we have
$$
J(u,dv)=\jmath(dv)
\Longrightarrow
C((u,\overline{u}),d(v,\overline{v}))=\jmath(dv)~\delta_v(d\overline{v}).
$$
In addition, consider some $2$-Wasserstein optimal coupling $\CC_{(u,v),(\overline{u},\overline{v})}((x,\overline{x}),d(y,\overline{y}))$ with marginals
$$
\BB_{u,v}(x,dy)\quad \mbox{\rm and}\quad\BB_{\overline{u},\overline{v}}(\overline{x},d\overline{y}).
$$
In this case, the product of transitions
$$
C((u,\overline{u}),d(v,\overline{v}))~\CC_{(u,v),(\overline{u},\overline{v})}((x,\overline{x}),d(y,\overline{y}))
$$
is a coupling with marginals $\MM((u,x),d(v,y))$ and $\MM((\overline{u},\overline{x}),d(\overline{v},\overline{y}))$. This implies that
\begin{multline*}
\Da_2\left(\MM((u,x),\point), \MM((\overline{u},\overline{x}),\point)\right)^2
\leq\Da_2\left(J(u,\point),J(\overline{u},\point)\right)^2\\
+\int C((u,\overline{u}),d(v,\overline{v}))~\Da_2\left(\BB_{u,v}(x,\point),\BB_{\overline{u},\overline{v}}(\overline{x},\point)\right)^2.
\end{multline*}
\emph{Case $u=\overline{u}$.} Coupling the labels identically, by (\ref{burg-def}) and $\kappa_{u,v}\leq\varsigma_v$ we have
$$
\Da_2\left(\MM((u,x),\point),\MM((u,\overline{x}),\point)\right)^2\leq \int J(u,dv)~\Vert \kappa_{u,v}~\cchi~(x-\overline{x})\Vert_2^2
\leq \Vert\cchi\Vert_2^2~\sup_{v\in\VV}\Vert\varsigma_v\Vert_2^2~\Vert x-\overline{x}\Vert_2^2 .
$$
This shows that (H\ref{ass:cont_cond_x}) is met for any label transition $J$, with
$$
a=\Vert\cchi\Vert_2~\sup_{v\in\VV}\Vert\varsigma_v\Vert_2,
$$
as soon as the latter is finite (the supremum being understood as a $\jmath$-essential supremum, since $J(u,\point)\ll\jmath$ for $\imath$-almost every $u$). Under (A\ref{ass-1}) below, one can take $a=\Vert\cchi\Vert_2~\Cu$.

\emph{Case $u\neq\overline{u}$.} Assume $J(u,dv)=\jmath(dv)$. Up to terms that are Lipschitz in $(x,\overline{x})$, the means of $\BB_{u,v}(x,\point)$ and $\BB_{\overline{u},v}(\overline{x},\point)$ differ by
$$
(\kappa_{u,v}-\kappa_{\overline{u},v})~\cchi~x .
$$
This term grows linearly in $x$, so that (H\ref{ass:cont_cond}) cannot hold unless $\kappa_{u,v}$ does not depend on $u$. This is the case for instance when $\sigma_u=\sigma$ for all $u\in\UU$; in this situation (H\ref{ass:cont_cond}) holds for finite $\UU$ and a suitable metric $\rho_{\UU}$.
\end{examp}

\subsection{A continuity theorem}

In practice the parameters of the Gaussian components are unknown and must be estimated, while the labels, their distributions $\imath,\jmath$ and the label transition $J$ are kept. Let
$$
(\widehat{p},\widehat{q},\widehat{P},\widehat{Q},\widehat{P}_{\imath},\widehat{Q}_{\jmath},\widehat{\BB}_{u,v})
$$
be the mathematical objects defined as $(p,q,P,Q,P_{\imath},Q_{\jmath},\BB_{u,v})$ by replacing the mappings $u\mapsto(m_u,\sigma_u)$ and $v\mapsto (n_v,\varsigma_v)$ by some estimators $u\mapsto (\widehat{m}_u,\widehat{\sigma}_u)$ and $v\mapsto(\widehat{n}_v,\widehat{\varsigma}_v)$. In particular, $\widehat{\BB}_{u,v}$ is the Gaussian Schr\"odinger bridge between the estimated components $\nu_{\widehat{m}_u,\widehat{\sigma}_u}$ and $\nu_{\widehat{n}_v,\widehat{\varsigma}_v}$. In this notation, the estimated distributions of the labelled states are given by
$$
\widehat{\PP}(d(u,x)):=\imath(du)~\nu_{\widehat{m}_u,\widehat{\sigma}_u}(dx)
\quad\mbox{\rm and}\quad
\widehat{\QQ}(d(v,y)):=\jmath(dv)~\nu_{\widehat{n}_v,\widehat{\varsigma}_v}(dy),
$$
and the estimated transport of the labelled states is given by the Markov transition $\widehat{\MM}$ from $\UU\times\RR^d$ into $\VV\times\RR^d$ defined by
$$
\widehat{\MM}((u,x),d(v,y)):=J(u,dv)~\widehat{\BB}_{u,v}(x,dy).
$$
In other words, $\widehat{\MM}$ moves the labels as $\MM$ does, and moves the states along the estimated bridges. As in Section~\ref{sec:product}, we check that $\widehat{\PP}\widehat{\MM}=\widehat{\QQ}$, and that the Markov transition
$$
\widehat{M}(x,dy):=\int_{(u,v)\in\UU\times\VV}\widehat{P}_{\imath}(x,du)~\widehat{\MM}((u,x),d(v,y))
$$
satisfies $\widehat{p}\widehat{M}=\widehat{q}$. We also have the bridge product formula, with the same label bridge $\pi\in\Ca(\imath,\jmath)$ defined in (\ref{pi-def}):
$$
\widehat{\PP}(d(u,x))~\widehat{\MM}((u,x),d(v,y))=\pi(d(u,v))~\nu_{\widehat{m}_u,\widehat{\sigma}_u}(dx)~\widehat{\BB}_{u,v}(x,dy).
$$
The following theorem is a direct consequence of Lemma~\ref{Lem-Tex}. It shows that the error of the estimated bridge is controlled by the error on the components, inflated by the Lipschitz constant of $\MM$, and by the averaged error on the component bridges.

\begin{theo}\label{thm:conti}
Assume (H\ref{ass:cont_cond_x}). Then we have that
\begin{multline}\label{conti-bound}
\Da_2\left(\widehat{\PP}\times\widehat{\MM},\PP\times\MM\right) \leq
(1+a^2)^{1/2}~\left(\int\imath(du)~\Da_2\left(\nu_{\widehat{m}_u,\widehat{\sigma}_u},\nu_{m_u,\sigma_u}\right)^2\right)^{1/2}\\
+\left(\int\pi(d(u,v))~\nu_{\widehat{m}_u,\widehat{\sigma}_u}(dx)~\Da_2\left(\widehat{\BB}_{u,v}(x,\point),\BB_{u,v}(x,\point)\right)^{2}\right)^{1/2}.
\end{multline}
\end{theo}
\begin{proof}
We couple $\widehat{\PP}$ and $\PP$ without changing the labels. For any $u\in\UU$, let $T_u$ be the Monge map (\ref{def-Monge}) pushing the estimated component $\nu_{\widehat{m}_u,\widehat{\sigma}_u}$ forward to $\nu_{m_u,\sigma_u}$, and $C_u$ the induced coupling
$$
C_u(d(x,\overline{x})):=\nu_{\widehat{m}_u,\widehat{\sigma}_u}(dx)~\delta_{T_u(x)}(d\overline{x}).
$$
By (\ref{eq:Monge}), $C_u$ is an optimal coupling of $\nu_{\widehat{m}_u,\widehat{\sigma}_u}$ and $\nu_{m_u,\sigma_u}$, and it depends measurably on $u$. The probability measure
$$
\gamma\left(d((u,x),(\overline{u},\overline{x}))\right):=\imath(du)~\delta_u(d\overline{u})~C_u(d(x,\overline{x}))
$$
is a coupling of $\widehat{\PP}$ and $\PP$, and we have
$$
\int \gamma\left(d((u,x),(\overline{u},\overline{x}))\right)~\rho_{\XX_d}((u,x),(\overline{u},\overline{x}))^2=\int\imath(du)~\Da_2\left(\nu_{\widehat{m}_u,\widehat{\sigma}_u},\nu_{m_u,\sigma_u}\right)^2 .
$$
Since $\gamma$ is supported on the set $\{u=\overline{u}\}$, condition (H\ref{ass:cont_cond_x}) ensures that (\ref{est-hyp}) holds $\gamma$-almost everywhere, with $\overline{K}=\MM$. Applying Lemma~\ref{Lem-Tex} with $(\mu,K)=(\widehat{\PP},\widehat{\MM})$ and $(\overline{\mu},\overline{K})=(\PP,\MM)$ we obtain the first term on the r.h.s.~of (\ref{conti-bound}), plus the term
$$
\left(\int\widehat{\PP}(d(u,x))~\Da_2\left(\widehat{\MM}((u,x),\point),\MM((u,x),\point)\right)^2\right)^{1/2}.
$$
Finally, the transitions $\widehat{\MM}((u,x),\point)$ and $\MM((u,x),\point)$ share the same label transition $J(u,dv)$. Coupling the labels identically, and $\widehat{\BB}_{u,v}(x,\point)$ with $\BB_{u,v}(x,\point)$ optimally, we find that
$$
\Da_2\left(\widehat{\MM}((u,x),\point),\MM((u,x),\point)\right)^2\leq\int J(u,dv)~\Da_2\left(\widehat{\BB}_{u,v}(x,\point),\BB_{u,v}(x,\point)\right)^2.
$$
The end of the proof is now a consequence of the identity
$$
\widehat{\PP}(d(u,x))~J(u,dv)=\pi(d(u,v))~\nu_{\widehat{m}_u,\widehat{\sigma}_u}(dx).
$$
This ends the proof of the theorem.
\end{proof}

\subsubsection{The case of unknown weights}\label{sec:unknown_weights}

If the probabilities $\imath$ and $\jmath$ are unknown (unknown weights), one can extend Theorem~\ref{thm:conti} in a natural way. We write the estimates of $\imath$, $\jmath$, $J$ and $\pi$ as $\widehat{\imath}$, $\widehat{\jmath}$, $\widehat{J}$ and $\widehat{\pi}$ respectively. In this context, the estimated distributions now carry the estimated weights and the estimated label transition, that is
$$
\widehat{\PP}(d(u,x))=\widehat{\imath}(du)~\widehat{P}(u,dx)
\quad\mbox{\rm and}\quad
\widehat{\MM}((u,x),d(v,y))=\widehat{J}(u,dv)~\widehat{\BB}_{u,v}(x,dy).
$$
In this situation, the labels are moved by the coupling of $\widehat\imath$ and $\imath$, and we need a continuity condition w.r.t.~the labels:

\begin{hypH}\label{ass:cont_cond_2}
There exists $a\geq 0$ such that for any $u,\overline{u}\in\UU$, $v,\overline{v}\in\VV$ and $x\in\RR^d$ we have that
\begin{eqnarray}
\Da_2\left(P(u,\point),P(\overline{u},\point)\right) & \leq & a~\rho_{\UU}(u,\overline{u}), \label{H22}\\
\Da_2\left(\BB_{u,v}(x,\point),\BB_{u,\overline{v}}(x,\point)\right) & \leq & a~\rho_{\VV}(v,\overline{v}).\label{H23}
\end{eqnarray}
\end{hypH}
For example, when $\UU$ is a finite set, one can easily check that (\ref{H22}) holds. When $\VV$ is a finite set, (\ref{H23}) holds as soon as $\kappa_{u,v}$ does not depend on $v$ (e.g.~for common covariances $\sigma_u=\sigma$ and $\varsigma_v=\varsigma$); otherwise the means of $\BB_{u,v}(x,\point)$ and $\BB_{u,\overline{v}}(x,\point)$ differ by a term growing linearly in $x$, and (\ref{H23}) cannot hold uniformly in $x$. This yields the following corollary; compared with Theorem~\ref{thm:conti}, the estimation of the weights adds two terms, one for $\widehat\imath$ and one for the label transition $\widehat J$.

\begin{cor}\label{cor:unknown}
Assume (H\ref{ass:cont_cond}--\ref{ass:cont_cond_2}). Then we have that
\begin{equation}\label{ext}
\begin{aligned}
\Da_2\left(\widehat{\PP}\times \widehat{\MM},\PP\times\MM \right)
\le\;& (1+a^2)^{1/2}\Bigl(\int \widehat{\imath}(du)\,
    \Da_2\bigl(\nu_{\widehat{m}_u,\widehat{\sigma}_u},\nu_{m_u,\sigma_u}\bigr)^2\Bigr)^{1/2}\\
&+\sqrt{2}\,(1+a^2)^{1/2}\Bigl(\int \widehat{\imath}(du)\,
    \Da_2\bigl(\widehat{J}(u,\point),J(u,\point)\bigr)^2\Bigr)^{1/2}
+(1+a^2)\,\Da_2\bigl(\widehat{\imath},\imath\bigr)\\
&+\sqrt{2}\,\Bigl(\int \widehat{\imath}(du)\,\widehat{J}(u,dv)\,
    \nu_{\widehat{m}_u,\widehat{\sigma}_u}(dx)\,
    \Da_2\bigl(\widehat{\BB}_{u,v}(x,\point),\BB_{u,v}(x,\point)\bigr)^{2}\Bigr)^{1/2}.
\end{aligned}
\end{equation}
\end{cor}
\begin{proof}
By (H\ref{ass:cont_cond}) and Lemma~\ref{Lem-Tex}, applied to the transitions $\widehat{\MM}$ and $\MM$, we have
$$
\Da_2\left(\widehat{\PP}\times\widehat{\MM},\PP \times \MM \right)
\le \sqrt{1+a^2}~\Da_2(\widehat{\PP},\PP)
+\Bigl(\int \widehat{\PP}(d(u,x))~\Da_2\bigl(\widehat{\MM}((u,x),\point),\MM((u,x),\point)\bigr)^2\Bigr)^{1/2}.
$$
By (\ref{H22}) and Lemma~\ref{Lem-Tex} we also have
$$
\Da_2(\widehat{\PP},\PP)\le \sqrt{1+a^2}~\Da_2(\widehat{\imath},\imath)+\left(\int\widehat{\imath}(du)~\Da_2(\widehat{P}(u,\point),P(u,\point))^2 \right)^{1/2}.
$$
Fix $(u,x)$. By (\ref{H23}) and Lemma~\ref{Lem-Tex}, using $(x+y)^2\le 2x^2+2y^2$, we find that
\begin{multline*}
\Da_2\bigl(\widehat{\MM}((u,x),\point),\MM((u,x),\point)\bigr)^2
\le 2(1+a^2)~\Da_2\bigl(\widehat{J}(u,\point),J(u,\point)\bigr)^2\\
+2\int \widehat{J}(u,dv)~\Da_2\bigl(\widehat{\BB}_{u,v}(x,\point),\BB_{u,v}(x,\point)\bigr)^2 .
\end{multline*}
Combining the above estimates and using $\sqrt{x+y}\le \sqrt{x}+\sqrt{y}$ yields (\ref{ext}). This ends the proof of the corollary.
\end{proof}

The remainder of this article primarily addresses the case of known weights; unknown weights are considered again in Section~\ref{sec:gauss_mix_em}.

\section{Stability Under Parameter Perturbation}\label{sec:stability}

In order to use the upper-bound in Theorem~\ref{thm:conti}, we need to control its two terms. The first one involves $2$-Wasserstein distances between Gaussian distributions, which can be controlled using (\ref{eqD2-g}). The second term on the r.h.s.~of (\ref{conti-bound}) is more challenging, as it requires the study of Riccati maps. We use an alternative form of the Riccati map, expressed in inverse coordinates, which is globally $1$-Lipschitz (Lemma~\ref{lem:h-lipschitz}). Our analysis culminates in Lemma~\ref{lem:w2-gauss}, which provides an upper-bound on the second term. We then combine Lemma~\ref{lem:w2-gauss} with standard estimates on $2$-Wasserstein distances between Gaussians to prove a more explicit upper-bound (Theorem~\ref{theo-main}).

\subsection{Riccati maps in inverse coordinates}\label{sec:ricc_maps}

The Riccati fixed point $r_{u,v}$ depends on the covariance matrices $(\sigma_u,\varsigma_v)$ only through the matrix parameter
$$
\psi(\sigma_u,\varsigma_v):=\varsigma_v^{1/2}~(\cchi~\sigma_u~\cchi^{\prime})~\varsigma_v^{1/2},
$$
or through its inverse
$$
\omega(\sigma_u,\varsigma_v):=\psi(\sigma_u,\varsigma_v)^{-1}=\varsigma_v^{-1/2}~\tau_{\beta}^{\prime}~\sigma_u^{-1}~\tau_{\beta}~\varsigma_v^{-1/2}
\quad\mbox{\rm with}\quad
\tau_{\beta}:=\beta^{-1}\tau .
$$
In this notation, we have
$$
\varpi_{u,v}=\omega(\sigma_u,\varsigma_v)
\quad\mbox{\rm and}\quad
r_{u,v}=\Ra(\omega(\sigma_u,\varsigma_v)),
$$
with the Riccati fixed point map
$$
\Ra(w):=\left(w+\frac{w^2}{4}\right)^{1/2}-\frac{w}{2},\qquad w\in\Sa_d^+ .
$$
Note that $\Ra(\omega)$ requires the inversion of $\psi$. It can however be expressed in terms of $w^{-1}$ only. Recalling that $(AB)^{1/2}=A^{1/2}B^{1/2}$ for commuting positive semi-definite matrices $A,B$, we check that
$$
\Ra(w)~\left(\left(w+\frac{w^2}{4}\right)^{1/2}+\frac{w}{2}\right)=w,
$$
as well as
$$
\Ra(w) = w~\left(\left(w+\frac{w^2}{4}\right)^{1/2}+\frac{w}{2}\right)^{-1}=
\left(\left(w^{-1}+\frac{I}{4}\right)^{1/2}+\frac{I}{2}\right)^{-1}.
$$
In the variable $\psi=w^{-1}$, the right-hand side defines a map on the whole set $\Sa^0_d$, including singular matrices. We introduce the rationalized map
\begin{equation}\label{def:h}
\psi\in\Sa_d^0\mapsto F(\psi):=\left(\left(\psi+\frac{I}{4}\right)^{1/2}+\frac{I}{2}\right)^{-1}
\quad\mbox{\rm and we set}\quad
\varphi(\sigma_u,\varsigma_v):=F\bigl(\psi(\sigma_u,\varsigma_v)\bigr).
\end{equation}
We summarize the above discussion with the following alternative description of the Riccati fixed point map.

\begin{prop}\label{prop:ricc-alt}
We have
$$
\Ra\bigl(\omega(\sigma_u,\varsigma_v)\bigr)=F\bigl(\psi(\sigma_u,\varsigma_v)\bigr).
$$
Consequently, we have that
\begin{equation}\label{ricc-infsup}
\left(\left(\lambda_{\max}(\psi)+\frac{1}{4}\right)^{1/2}+\frac{1}{2}\right)^{-1} I\le \Ra(\omega)= F(\psi)\leq I.
\end{equation}
\end{prop}
\begin{proof}
The first assertion follows from the above display, with $w=\omega=\psi^{-1}$. The estimates (\ref{ricc-infsup}) follow from
\[
\displaystyle I\le\left(\psi+\frac{I}{4}\right)^{1/2}+\frac{I}{2}\le\left(\left(\lambda_{\max}(\psi)+\frac14\right)^{1/2}+\frac12\right)I .
\]
This ends the proof of the proposition.
\end{proof}

The rationalized Riccati map (\ref{def:h}) has the following $1$-Lipschitz property. The proof of the next lemma is rather technical, thus it is provided in the appendix, see Section~\ref{app:lip-ricatti}. Note that the constant $1$ is sharp, since the derivative of the scalar version of $F$ at the origin is equal to $-1$.

\begin{lem}\label{lem:h-lipschitz}
For any $\psi_1,\psi_2\in\Sa_d^0$ and any unitarily invariant matrix norm $\Vert\point\Vert$, we have
\begin{equation}\label{h-lipschitz}
\Vert F(\psi_1)-F(\psi_2)\Vert\leq\Vert\psi_1-\psi_2\Vert .
\end{equation}
\end{lem}

Consider the covariance approximations $(\widehat{\sigma}_u, \widehat{\varsigma}_v)$. In this notation, we have the following local Lipschitz estimate. The proof follows elementary but lengthy calculations, thus it is provided in the appendix, see Section~\ref{app:lip}.

\begin{lem}\label{lem:psi}
We have that
\begin{multline*}
\Vert\psi(\widehat{\sigma}_u,\widehat{\varsigma}_v)-\psi(\sigma_u,\varsigma_v)\Vert_2
\leq \frac{\Vert \cchi\Vert_2^2}{\lambda^{1/2}_{\min}(\varsigma_v)}~
\Vert\widehat{\sigma}_u\Vert_2\left(\Vert\varsigma_v\Vert_2^{1/2}+\Vert\widehat{\varsigma}_v\Vert^{1/2}_2\right)~\Vert \widehat{\varsigma}_v-\varsigma_v\Vert_2\\
+\Vert \cchi\Vert_2^2~\Vert\varsigma_v\Vert_2~\Vert \widehat{\sigma}_u-\sigma_u\Vert_2 .
\end{multline*}
\end{lem}

\subsection{Gaussian bridges}

We now compare the exact and the estimated bridges, started from the same state $x$. Both are Gaussian. The exact bridge $\BB_{u,v}(x,\point)$ is the distribution of the random map $Y_{u,v}(x)$ defined in (\ref{def-Y}); by Proposition~\ref{prop:ricc-alt}, we have $\kappa_{u,v}=\varsigma_v^{1/2}~\varphi(\sigma_u,\varsigma_v)~\varsigma_v^{1/2}$. In the same vein, the estimated bridge $\widehat{\BB}_{u,v}(x,\point)$ is the distribution of the random map
$$
\widehat{Y}_{u,v}(x):= \widehat{n}_v+\widehat{\kappa}_{u,v}~\cchi~(x-\widehat{m}_u)+\widehat{\kappa}_{u,v}^{1/2}~\widehat{G}
\quad\mbox{\rm with}\quad \widehat{\kappa}_{u,v}:=\widehat{\varsigma}_v^{1/2}~\varphi(\widehat{\sigma}_u,\widehat{\varsigma}_v)~\widehat{\varsigma}_v^{1/2},
$$
where $\widehat{G}$ stands for a $d$-dimensional standard Gaussian random variable, independent of $G$. Using (\ref{ricc-infsup}) we check that
\begin{equation}\label{ricc-infsup-2}
\kappa_{u,v}\leq \varsigma_v
\quad\mbox{\rm and}\quad
\widehat{\kappa}_{u,v}\leq \widehat{\varsigma}_v,
\quad\mbox{\rm so that}\quad
\Vert \kappa_{u,v} \Vert _2\leq  \Vert \varsigma_v \Vert_2
\quad\mbox{\rm and}\quad
\Vert  \widehat{\kappa}_{u,v}\Vert_2\leq  \Vert \widehat{\varsigma}_v\Vert_2.
\end{equation}
The two Gaussian distributions have covariance matrices $\kappa_{u,v}$ and $\widehat\kappa_{u,v}$, and the difference of their means is given by
$$
\left(n_v-\widehat{n}_v+\kappa_{u,v}~\cchi~(\widehat{m}_u-m_u)\right)+(\kappa_{u,v}-\widehat{\kappa}_{u,v})~\cchi~(x-\widehat{m}_u).
$$
Observe that the second term is centred w.r.t.~$\nu_{\widehat{m}_u,\widehat{\sigma}_u}(dx)$, so that the cross term vanishes after integration. In addition, we have
\begin{eqnarray*}
\int\nu_{\widehat{m}_u,\widehat{\sigma}_u}(dx)~\Vert(\kappa_{u,v}-\widehat{\kappa}_{u,v})~\cchi~(x-\widehat{m}_u)\Vert_2^2
&=&\tr\left((\kappa_{u,v}-\widehat{\kappa}_{u,v})~\cchi~\widehat{\sigma}_u~\cchi^{\prime}~(\kappa_{u,v}-\widehat{\kappa}_{u,v})\right)\\
&\leq& \Vert \widehat{\sigma}_u\Vert_2~\Vert \cchi\Vert_2^2~\Vert \kappa_{u,v}-\widehat{\kappa}_{u,v}\Vert_F^2 .
\end{eqnarray*}
Combining these observations with (\ref{eqD2-g}) and (\ref{ricc-infsup-2}), we readily check the following lemma.

\begin{lem}\label{lem:w2-gauss}
We have that
\begin{equation}\label{ineq:w2-gauss}
\begin{split}
\int\nu_{\widehat{m}_u,\widehat{\sigma}_u}(dx)~\Da_2(\BB_{u,v}(x,\point),\widehat{\BB}_{u,v}(x,\point))^2
\leq{}& \left(\frac{1}{\lambda_{\rm min}(\kappa_{u,v})}+\Vert \widehat{\sigma}_u\Vert_2~\Vert \cchi\Vert_2^2\right)\Vert \kappa_{u,v}-\widehat{\kappa}_{u,v}\Vert_F^2\\
&+2\Vert n_v-\widehat{n}_v\Vert_2^2+2\Vert \varsigma_v \Vert_2^2~\Vert\cchi\Vert_2^2~\Vert\widehat{m}_u-m_u\Vert_2^2.
\end{split}
\end{equation}
\end{lem}

To take the final step, we need to estimate the difference $\kappa_{u,v}-\widehat{\kappa}_{u,v}$. Using the decomposition
\begin{multline*}
\kappa_{u,v}-\widehat{\kappa}_{u,v}
=(\varsigma_v^{1/2}-\widehat{\varsigma}_v^{1/2})~\varphi(\sigma_u,\varsigma_v)~\varsigma_v^{1/2}+
\widehat{\varsigma}_v^{1/2}~(\varphi(\sigma_u,\varsigma_v)-\varphi(\widehat{\sigma}_u,\widehat{\varsigma}_v))~\varsigma_v^{1/2}\\+
\widehat{\varsigma}_v^{1/2}~\varphi(\widehat{\sigma}_u,\widehat{\varsigma}_v)~(\varsigma_v^{1/2}-\widehat{\varsigma}_v^{1/2}),
\end{multline*}
the estimate (\ref{square-root-key-estimate}) and the fact that $\Vert\varphi\Vert_2\leq 1$, we check that
\begin{equation}\label{ineq:spe}
\Vert\kappa_{u,v}-\widehat{\kappa}_{u,v}\Vert_2\leq \frac{\Vert\varsigma_v\Vert_2^{1/2}+\Vert\widehat{\varsigma}_v\Vert^{1/2}_2}{\lambda^{1/2}_{\rm min}(\varsigma_v)}~\Vert\varsigma_v-\widehat{\varsigma}_v\Vert_2 +
\Vert\widehat{\varsigma}_v\Vert^{1/2}_2~\Vert \varsigma_v\Vert^{1/2}_2~
\Vert\varphi(\sigma_u,\varsigma_v)-\varphi(\widehat{\sigma}_u,\widehat{\varsigma}_v)\Vert_2.
\end{equation}

\subsection{Stability of Gaussian mixture bridges}

For the sake of brevity, we assume $J(u,dv)=\jmath(dv)$; the general case is easy to derive and only changes the moment terms below. Let $(\widehat{m}_u,\widehat{\sigma}_u)$ and $(\widehat{n}_v,\widehat{\varsigma}_v)$ be some estimates of the parameter mappings $u\mapsto(m_u,\sigma_u)$ and $v\mapsto(n_v,\varsigma_v)$. We quantify the estimation errors on the means and the covariances by the following averaged quantities, with integer $k\geq 1$,
\begin{equation}\label{def-deltas}
\begin{array}{ll}
\displaystyle\Delta_{m,k}:=\int\imath(du)~\Vert\widehat{m}_u-m_u\Vert_2^k,
&\displaystyle\Delta_{\sigma,k}:=\int\imath(du)~\Vert\widehat{\sigma}_u-\sigma_u\Vert_2^k,\\[8pt]
\displaystyle\Delta_{n,k}:=\int\jmath(dv)~\Vert\widehat{n}_v-n_v\Vert_2^k,
&\displaystyle\Delta_{\varsigma,k}:=\int\jmath(dv)~\Vert\widehat{\varsigma}_v-\varsigma_v\Vert_2^k.
\end{array}
\end{equation}
We also measure the size of the estimated covariances by
\begin{equation}\label{def-b}
b_{\sigma,k}:=\int\imath(du)~\Vert\widehat{\sigma}_u\Vert_2^k
\quad\mbox{\rm and}\quad
{b}_{\varsigma,k}:=\int \jmath(dv)~\Vert\widehat{\varsigma}_v\Vert_2^k.
\end{equation}
By the Cauchy-Schwarz inequality, for any integers $k,l \geq 1$ we have
\begin{equation}\label{Cauchy}
\int\imath(du)~\Vert\widehat{\sigma}_u\Vert_2^k~\Vert\widehat{\sigma}_u-\sigma_u\Vert_2^l \le b_{\sigma,2k}^{1/2}~\Delta_{\sigma,2l}^{1/2}
\quad\mbox{\rm and}\quad
\int\jmath(dv)~\Vert\widehat{\varsigma}_v\Vert_2^k~\Vert\widehat{\varsigma}_v-\varsigma_v\Vert_2^l \le b_{\varsigma,2k}^{1/2}~\Delta_{\varsigma,2l}^{1/2}.
\end{equation}
All the products of these moments arising in the proof of Theorem~\ref{theo-main} below are collected in the parameter
\begin{equation}\label{bmax}
\texttt{B}:=\max\left\{1,\, b_{\varsigma,2}^{1/2},\, b_{\sigma,2} b_{\varsigma,2}^{1/2},\, b_{\sigma,2} b_{\varsigma,4}^{1/2},\, b_{\sigma,1},\, b_{\sigma,1} b_{\varsigma,2}^{1/2},\, b_{\varsigma,1} b_{\sigma,2}^{1/2},\, b_{\sigma,3} b_{\varsigma,2}^{1/2},\, b_{\sigma,3} b_{\varsigma,4}^{1/2} \right\}.
\end{equation}
We work under the following uniform condition, which ensures that the exact covariance matrices are well conditioned.

\begin{hypA}\label{ass-1}
There exist $0<\Cl\leq \Cu<\infty$ such that for $\imath$-almost every $u\in\UU$ and $\jmath$-almost every $v\in\VV$ we have
$$
\lambda_{\min}(\sigma_u)\wedge\lambda_{\min}(\varsigma_v)\geq \Cl
\quad\mbox{\rm and}\quad
\Vert\sigma_u\Vert_2\vee\Vert\varsigma_v\Vert_2 \leq \Cu.
$$
\end{hypA}
Under (A\ref{ass-1}), the following bounds are uniform in $(u,v)$:
$$
\lambda_{\max}\left(\psi(\sigma_u,\varsigma_v)\right)\leq \psi_1:=\Cu^2~\Vert\cchi\Vert_2^2,
$$
and by (\ref{ricc-infsup}) we have
\begin{equation}\label{kappa-min}
\lambda_{\min}(\kappa_{u,v})\geq \kappa_0:=\Cl~\left(\frac12+\left(\psi_1+\frac{1}{4}\right)^{1/2}\right)^{-1}
\quad\mbox{\rm and}\quad
\Vert\kappa_{u,v}\Vert_2\leq \Vert\varsigma_v\Vert_2\leq \Cu .
\end{equation}
The following theorem shows that, under (A\ref{ass-1}), the squared $2$-Wasserstein error is linear in the averaged squared estimation errors on the means and the covariances, up to a factor $d$ coming from the Frobenius norms.

\begin{theo}\label{theo-main}
Assume (A\ref{ass-1}). Then there exists a $C<\infty$ depending only on $(\Cl,\Cu,\cchi)$ such that for any $d\geq 1$ we have that
\begin{equation}\label{main-bound}
\Da_2\left(\widehat{\PP}\times\widehat{\MM},\PP\times\MM\right)^2
\leq
Cd\left[
\Delta_{m,2}+\Delta_{n,2}+\Delta_{\sigma,2}+\texttt{B}\left(\Delta_{\varsigma,2}+\Delta_{\sigma,2}+\Delta_{\sigma,4}^{1/2}+\Delta_{\varsigma,4}^{1/2}\right)
\right],
\end{equation}
where $\texttt{B}$ is defined in (\ref{bmax}).
\end{theo}

\begin{proof}
In the proof, $C$ denotes a finite constant depending only on $(\Cl,\Cu,\cchi)$, whose value may change from line to line. By Example~\ref{ex}, condition (H\ref{ass:cont_cond_x}) holds with $a=\Vert\cchi\Vert_2~\Cu$ under (A\ref{ass-1}). Hence, by Theorem~\ref{thm:conti} and the inequality $(x+y)^2\leq 2x^2+2y^2$, we have
\begin{equation}\label{eq:main_bound_prf}
\begin{split}
\Da_2\left(\widehat{\PP}\times\widehat{\MM},\PP\times\MM\right)^2
\leq{}&2(1+a^2)\int\imath(du)~\Da_2\left(\nu_{\widehat{m}_u,\widehat{\sigma}_u},\nu_{m_u,\sigma_u}\right)^2\\
&+2\int\pi(d(u,v))~\nu_{\widehat{m}_u,\widehat{\sigma}_u}(dx)~\Da_2\left(\widehat{\BB}_{u,v}(x,\point),\BB_{u,v}(x,\point)\right)^2 .
\end{split}
\end{equation}
We now estimate the two terms on the r.h.s.~of the above inequality. The first one is rather elementary; using (\ref{eqD2-g}) and (A\ref{ass-1}) we check that
$$
\Da_2\left(\nu_{\widehat{m}_u,\widehat{\sigma}_u},\nu_{m_u,\sigma_u}\right)^2
\leq
\frac{1}{\lambda_{\min}(\sigma_u)}~\Vert\widehat{\sigma}_u-\sigma_u\Vert_F^2+\Vert\widehat{m}_u-m_u\Vert_2^2
\leq
\Cl^{-1}~\Vert\widehat{\sigma}_u-\sigma_u\Vert_F^2+\Vert\widehat{m}_u-m_u\Vert_2^2,
$$
which integrates into $\Cl^{-1}d~\Delta_{\sigma,2}+\Delta_{m,2}$.

Now turning to the second term on the r.h.s.~of (\ref{eq:main_bound_prf}), by Lemma~\ref{lem:w2-gauss} and the estimate (\ref{kappa-min}) we have
\begin{equation}\label{eq:BB}
\begin{split}
\int\nu_{\widehat{m}_u,\widehat{\sigma}_u}(dx)~\Da_2\left(\BB_{u,v}(x,\point),\widehat{\BB}_{u,v}(x,\point)\right)^2
\leq{}&
\left(\kappa_0^{-1}+\Vert\widehat{\sigma}_u\Vert_2~\Vert\cchi\Vert_2^2\right)\Vert\kappa_{u,v}-\widehat{\kappa}_{u,v}\Vert_F^2\\
&+2\Vert n_v-\widehat{n}_v\Vert_2^2
+2\Cu^2~\Vert\cchi\Vert_2^2~\Vert\widehat{m}_u-m_u\Vert_2^2 .
\end{split}
\end{equation}
By (\ref{ineq:spe}) and Lemma~\ref{lem:h-lipschitz} we also have
$$
\Vert\kappa_{u,v}-\widehat{\kappa}_{u,v}\Vert_2
\leq
\frac{\sqrt{\Cu}+\Vert\widehat{\varsigma}_v\Vert_2^{1/2}}{\sqrt{\Cl}}~\Vert\varsigma_v-\widehat{\varsigma}_v\Vert_2
+\sqrt{\Cu}~\Vert\widehat{\varsigma}_v\Vert_2^{1/2}~\Vert\psi(\widehat{\sigma}_u,\widehat{\varsigma}_v)-\psi(\sigma_u,\varsigma_v)\Vert_2,
$$
and by Lemma~\ref{lem:psi} and (A\ref{ass-1})
$$
\Vert\psi(\widehat{\sigma}_u,\widehat{\varsigma}_v)-\psi(\sigma_u,\varsigma_v)\Vert_2
\leq
\frac{\Vert\cchi\Vert_2^2}{\sqrt{\Cl}}~\Vert\widehat{\sigma}_u\Vert_2\left(\sqrt{\Cu}+\Vert\widehat{\varsigma}_v\Vert_2^{1/2}\right)
\Vert\widehat{\varsigma}_v-\varsigma_v\Vert_2
+\Vert\cchi\Vert_2^2~\Cu~\Vert\widehat{\sigma}_u-\sigma_u\Vert_2 .
$$
Squaring and using $(x+y)^2\leq 2x^2+2y^2$ repeatedly, together with the norm equivalence $\Vert\point\Vert_F\leq\sqrt{d}~\Vert\point\Vert_2$, we arrive at the pointwise estimate
\begin{multline*}
\Vert\kappa_{u,v}-\widehat{\kappa}_{u,v}\Vert_F^2
\leq
Cd~\Big[(1+\Vert\widehat{\varsigma}_v\Vert_2)~\Vert\widehat{\varsigma}_v-\varsigma_v\Vert_2^2\\
+\Vert\widehat{\varsigma}_v\Vert_2\left(\Vert\widehat{\sigma}_u\Vert_2^2~(1+\Vert\widehat{\varsigma}_v\Vert_2)~\Vert\widehat{\varsigma}_v-\varsigma_v\Vert_2^2+\Vert\widehat{\sigma}_u-\sigma_u\Vert_2^2\right)\Big].
\end{multline*}
Integrating w.r.t.~$\pi(d(u,v))=\imath(du)\jmath(dv)$, with (\ref{def-b}) and the Cauchy-Schwarz inequality (\ref{Cauchy}), we obtain
\begin{equation}\label{int-kappa}
\int\pi(d(u,v))~\Vert\kappa_{u,v}-\widehat{\kappa}_{u,v}\Vert_F^2
\le
Cd\left(\Delta_{\varsigma,2}+b_{\varsigma,2}^{1/2}\Delta_{\varsigma,4}^{1/2}+b_{\varsigma,1}\Delta_{\sigma,2}+b_{\sigma,2}b_{\varsigma,2}^{1/2}\Delta_{\varsigma,4}^{1/2}+b_{\sigma,2}b_{\varsigma,4}^{1/2}\Delta_{\varsigma,4}^{1/2}\right).
\end{equation}
Similarly, we also have
\begin{equation}\label{int-kappa-2}
\begin{split}
\int\pi(d(u,v))~\Vert\widehat{\sigma}_u\Vert_2~\Vert\kappa_{u,v}-\widehat{\kappa}_{u,v}\Vert_F^2
\le{}&
Cd\bigl(b_{\sigma,1}\Delta_{\varsigma,2}+b_{\sigma,1}b_{\varsigma,2}^{1/2}\Delta_{\varsigma,4}^{1/2}+b_{\varsigma,1}b_{\sigma,2}^{1/2}\Delta_{\sigma,4}^{1/2}\\
&+b_{\sigma,3}b_{\varsigma,2}^{1/2}\Delta_{\varsigma,4}^{1/2}+b_{\sigma,3}b_{\varsigma,4}^{1/2}\Delta_{\varsigma,4}^{1/2}\bigr).
\end{split}
\end{equation}
Note that $1\le\texttt{B}$ and $b_{\varsigma,1}\le b_{\varsigma,2}^{1/2}\le\texttt{B}$. Plugging (\ref{int-kappa}) and (\ref{int-kappa-2}) into (\ref{eq:BB}), and integrating the mean terms into $\Delta_{n,2}$ and $\Delta_{m,2}$, yields (\ref{main-bound}). This ends the proof of the theorem.
\end{proof}

Several remarks are of interest here.
\begin{itemize}
\item The factor $d$ in (\ref{main-bound}) comes from the Frobenius norms in (\ref{eqD2-g}) and in (\ref{ineq:w2-gauss}); all the other estimates are expressed in terms of spectral norms, and the constant $C$ only depends on $(\Cl,\Cu,\cchi)$.
\item The parameter $\texttt{B}$ only depends on the moments (\ref{def-b}) of the estimated covariance matrices. In the context of the empirical estimates discussed in Section~\ref{sec:gb}, its second moment is uniformly bounded w.r.t.~the number of samples $N$ and the dimension $d$, as soon as $N>d$.
\item For general label transitions $J(u,dv)$, the same line of arguments applies by replacing the product measure $\imath\otimes\jmath$ by the label bridge $\pi$ defined in (\ref{pi-def}); in this situation, the products of moments in (\ref{int-kappa}) and (\ref{int-kappa-2}) are replaced by Cauchy-Schwarz estimates w.r.t.~$\pi$.
\end{itemize}

\section{Finite-Sample Convergence}\label{sec:empirical}

We now consider bounds when data are used to estimate the parameters of the Gaussians. We consider three cases: Gaussian mixtures; a single Gaussian on each side of the bridge; and the regularized empirical Monge map.

\subsection{Gaussian mixtures}\label{sec:gauss_mix_em}

In practice, when one only has samples from Gaussian mixtures, one needs specific methods to estimate the parameters. We illustrate the generality of our results in the context of the EM algorithm. The following result is essentially an application of Theorem~\ref{theo-main} combined with \cite[Theorem 8]{bing}, which provides several results for the EM algorithm and the estimation of Gaussian mixtures. We remark that in the context described below, the mixture weights are unknown and estimated, as discussed in Section~\ref{sec:unknown_weights}.

We briefly recall the set-up in \cite{bing}. Let $\UU=\{1,\dots,k_U\}$ and $\VV=\{1,\dots,k_V\}$ be finite sets, and assume common component covariances: $\sigma_u=\sigma$ for all $u\in\UU$ and $\varsigma_v=\varsigma$ for all $v\in\VV$. We write
$$
\imath_u:=\imath(\{u\}),\qquad {\jmath}_v:=\jmath(\{v\}),\qquad \underline{\imath}:=\min_u \imath_u>0
\quad\mbox{\rm and}\quad \underline{\jmath}:=\min_v\jmath_v>0.
$$
For the sake of brevity, we assume $J(u,dv)=\jmath(dv)$ as above. Let $(\widehat{\imath},\widehat{m}_u,\widehat{\sigma})$ and $(\widehat{\jmath},\widehat{n}_v,\widehat{\varsigma})$ be the EM iterates after $t\geq \texttt{C}\log N$ steps (see $\texttt{C}$ in \cite[(3.4)]{bing}), each based on $N$ i.i.d.~samples from the respective mixture. The article \cite{bing} proves that, under a separation condition on the mixture and a good initialization of the EM algorithm, all the estimates satisfy non-asymptotic bounds of order $(d\log N/N)^{1/2}$. In the following result, we assume that \cite[(3.4), (3.5) and (4.7)]{bing} hold at both ends of the pair of mixtures, and that
$$
Nd \ge \max(k_U,k_V)^2\log^3N,\qquad d\log N\leq N
\quad\mbox{\rm and}\quad N>10.
$$
On the event where the EM algorithm succeeds, the weights contribute a term of order $(d\log N/N)^{1/2}$ to the squared error, and the components a term of order $d^2\log N/N$.

\begin{cor}\label{cor:em}
Assume (H\ref{ass:cont_cond}--\ref{ass:cont_cond_2}) and (A\ref{ass-1}). Then, up to a permutation of the labels, on an event of probability at least $1-10N^{-1}$, we have
\begin{equation}\label{ineq:mix-bridge}
\Da_2\big(\widehat{\PP}\times\widehat{\MM},\PP\times\MM\big)^2
\le C \left[\left(\frac{d\log N}{N}\right)^{1/2}+\frac{d^2\log N}{N} \right],
\end{equation}
for some $C<\infty$ depending only on $(a,\Cl,\Cu,\cchi,\underline{\imath}, \underline{\jmath},k_U,k_V,\texttt{C})$ and on the constants in \cite[(3.5)]{bing}.
\end{cor}

\begin{proof}
By (\ref{Mnorm}) and \cite[Theorem 8]{bing}, at each end, with probability at least $1-5N^{-1}$, there exists a constant $C<\infty$ such that
\begin{align}
\|\widehat{\imath}-\imath\|_1 &\le C\sqrt{\frac{d\log N}{N\underline{\imath}}},\label{em-i}\\
\max_{u\in\UU}\|\widehat{m}_u-m_u\|_2^2 &\leq \lambda_{\max}(\sigma)~\max_{u\in\UU}\Vert\widehat{m}_u-m_u\Vert^2_{\sigma}
\le C~\frac{d\log N}{N\underline{\imath}},\label{em-m}\\
\|\widehat{\sigma}-\sigma\|_2 &\leq \lambda_{\max}(\sigma)~\Vert\sigma^{-1/2}(\widehat{\sigma}-\sigma)\sigma^{-1/2}\Vert_2
\le  C\sqrt{\frac{d\log N}{N}}.\label{em-s}
\end{align}
Since $\UU$ is finite, choosing the maximal coupling $\gamma \in \Ca(\widehat{\imath},\imath)$ we find that
\begin{equation}\label{em-wei}
\Da_2\big(\widehat{\imath},\imath\big)^2\le \int \rho_{\UU}(u,\overline{u})^2~\gamma(d(u,\overline{u}))
\leq \Big(\sup_{u,\overline{u}\in \UU}\rho_{\UU}(u,\overline{u})\Big)^2\int 1_{u \neq \overline{u}}~\gamma(d(u,\overline{u}))
\le C~\|\widehat{\imath}-\imath\|_1.
\end{equation}
Similar results hold for $(\widehat{\jmath},\widehat{n}_v,\widehat{\varsigma})$. In the rest of the proof we work on the intersection of the two events, which has probability at least $1-10N^{-1}$. On this event, since $d\log N\leq N$, the estimate (\ref{em-s}) and (A\ref{ass-1}) yield
$$
\Vert\widehat{\sigma}\Vert_2\vee\Vert\widehat{\varsigma}\Vert_2\leq C.
$$
We now estimate the four terms on the r.h.s.~of (\ref{ext}), recalling that $J(u,\point)=\jmath$ and $\widehat{J}(u,\point)=\widehat{\jmath}$. By (\ref{eqD2-g}), (A\ref{ass-1}), (\ref{em-m}), (\ref{em-s}) and the norm equivalence $\Vert\point\Vert_F\leq\sqrt{d}~\Vert\point\Vert_2$ we have
$$
\int\widehat{\imath}(du)~\Da_2\left(\nu_{\widehat{m}_u,\widehat{\sigma}},\nu_{m_u,\sigma}\right)^2\leq\max_{u\in\UU}\Vert\widehat{m}_u-m_u\Vert_2^2+\frac{d}{\Cl}~\Vert\widehat{\sigma}-\sigma\Vert_2^2\leq C~\frac{d^2\log N}{N}.
$$
By (\ref{em-wei}) and (\ref{em-i}) we also have
$$
\Da_2(\widehat{\imath},\imath)^2+\Da_2(\widehat{\jmath},\jmath)^2\leq C~\left(\frac{d\log N}{N}\right)^{1/2}.
$$
Finally, the pointwise estimates (\ref{ineq:w2-gauss}) and (\ref{ineq:spe}), Lemma~\ref{lem:h-lipschitz} and Lemma~\ref{lem:psi}, as in the proof of Theorem~\ref{theo-main}, combined with the boundedness of $\Vert\widehat{\sigma}\Vert_2$ and $\Vert\widehat{\varsigma}\Vert_2$ on the event, yield the estimate
$$
\begin{aligned}
\int\nu_{\widehat{m}_u,\widehat{\sigma}}(dx)~\Da_2\left(\widehat{\BB}_{u,v}(x,\point),\BB_{u,v}(x,\point)\right)^2
&\leq C~d\left(\Vert\widehat{\sigma}-\sigma\Vert_2^2+\Vert\widehat{\varsigma}-\varsigma\Vert_2^2\right)+C\left(\Vert\widehat{m}_u-m_u\Vert_2^2+\Vert\widehat{n}_v-n_v\Vert_2^2\right)\\
&\leq C~\frac{d^2\log N}{N},
\end{aligned}
$$
uniformly in $(u,v)$. Squaring (\ref{ext}) and using
$$
\Big(\sum_{1\leq i\leq 4} x_i\Big)^2\leq 4\sum_{1\leq i\leq 4}x_i^2,
$$
we obtain (\ref{ineq:mix-bridge}). This ends the proof of the corollary.
\end{proof}

\begin{rmk}
The first term in (\ref{ineq:mix-bridge}) is generated by the mixture weight errors. One can also use the mixture Wasserstein distance \cite{delon} to analyze this bound; the resulting rate is the same.
\end{rmk}

\subsection{The empirical Gaussian bridge}\label{sec:gb}

We now specialize to a single pair of Gaussian marginals, where the estimation error is produced by sampling. Consider the point-mass case $\imath=\delta_u$ and $\jmath=\delta_v$, and write $m=m_u$, $\sigma=\sigma_u$, $n=n_v$, $\varsigma=\varsigma_v$, $\kappa=\kappa_{u,v}$ and so on. In this situation, the product bridge $\PP\times\MM$ identifies with the Schr\"odinger bridge coupling $\pi^\star$ between $\nu_{m,\sigma}$ and $\nu_{n,\varsigma}$ on $\RR^d\times\RR^d$, and $\widehat{\PP}\times\widehat{\MM}$ with the empirical bridge $\widehat{\pi}^\star$ built from the sample moments
\begin{equation}\label{sample}
\widehat{m}=\frac{1}{N+1}\sum_{1\leq i\leq N+1}X_i
\quad\mbox{\rm and}\quad
\widehat{\sigma}=\frac{1}{N}\sum_{1\leq i\leq N+1}(X_i-\widehat{m})(X_i-\widehat{m})^{\prime},
\end{equation}
and the analogous estimates $(\widehat{n},\widehat{\varsigma})$. Here $(X_i)_{1\leq i\leq N+1}$ and the sample used for $(\widehat{n},\widehat{\varsigma})$ are two independent collections of $N+1$ independent and identically distributed (i.i.d.) samples from $\nu_{m,\sigma}$ and $\nu_{n,\varsigma}$, respectively. We recall some rather well known non-asymptotic estimates of the sample moments. For the convenience of the reader, a short proof is provided below.

\begin{lem}\label{lem:wishart}
Let $N>d\geq 1$. Then the following assertions hold.
\begin{enumerate}
\item[(i)] $\widehat{m}$ and $\widehat{\sigma}$ are independent, and
$$
\widehat{m}-m\sim\nu_{0,\sigma/(N+1)}
\quad\mbox{\rm and}\quad
N\widehat{\sigma}\stackrel{\rm law}{=}\sigma^{1/2}~ZZ^{\prime}~\sigma^{1/2},
$$
with $Z$ a $d\times N$ matrix of i.i.d.~standard Gaussian entries. In particular, for any $k\geq 1$ we have
$$
\EE\left[\Vert\widehat{m}-m\Vert_2^{2k}\right]\leq C_k\left(\frac{\tr(\sigma)}{N}\right)^k.
$$
\item[(ii)] For any $\eta>0$ such that $\sqrt{d/N}+\eta\leq 1$ we have
$$
\Prob\left(\Vert\widehat{\sigma}-\sigma\Vert_2>3~\Vert\sigma\Vert_2\left(\sqrt{d/N}+\eta\right)\right)\leq 2~e^{-N\eta^2/2}.
$$
\item[(iii)] For any $k\geq 1$ there exists a constant $C_k<\infty$ depending only on $k$ such that
$$
\EE\left[\Vert\widehat{\sigma}-\sigma\Vert_2^k\right]\leq C_k~\Vert\sigma\Vert_2^k~(d/N)^{k/2}
\quad\mbox{\rm and}\quad
\EE\left[\Vert\widehat{\sigma}\Vert_2^k\right]\leq C_k~\Vert\sigma\Vert_2^k.
$$
\end{enumerate}
\end{lem}
\begin{proof}
Assertion (i) is classical, see for instance \cite[Chapter 3]{anderson}; the moment bound follows from Gaussian moment equivalence. For (ii) and (iii), by \cite[Corollary 5.35]{vershynin} applied to the $N\times d$ matrix $Z^{\prime}$, for any $t\geq 0$, with probability at least $1-2e^{-t^2/2}$ we have
$$
\sqrt{N}-\sqrt{d}-t\leq s_{\min}(Z^{\prime})\leq s_{\max}(Z^{\prime})\leq\sqrt{N}+\sqrt{d}+t.
$$
Equivalently, the random variable
$$
\Theta:=\max\left\{0,~s_{\max}(Z^{\prime})-\sqrt{N}-\sqrt{d},~\sqrt{N}-\sqrt{d}-s_{\min}(Z^{\prime})\right\}
$$
satisfies $\Prob(\Theta>t)\leq 2e^{-t^2/2}$. Set $\delta:=(\sqrt{d}+\Theta)/\sqrt{N}$. The eigenvalues of $N^{-1}ZZ^{\prime}$ lie in the interval $[((1-\delta)_+)^2,(1+\delta)^2]$, and since $N^{-1}ZZ^{\prime}\geq0$ we check that
$$
\Vert N^{-1}ZZ^{\prime}-I\Vert_2\leq 2\delta+\delta^2,
$$
with $2\delta+\delta^2\leq 3\delta$ when $\delta\leq1$. On the other hand, we have
$$
\Vert\widehat{\sigma}-\sigma\Vert_2\stackrel{\rm law}{=}\Vert\sigma^{1/2}(N^{-1}ZZ^{\prime}-I)\sigma^{1/2}\Vert_2\leq\Vert\sigma\Vert_2~\Vert N^{-1}ZZ^{\prime}-I\Vert_2,
$$
so that (ii) follows by choosing $t=\eta\sqrt{N}$. For (iii), the Gaussian tail of $\Theta$ yields $\EE[\Theta^p]\leq C_p$, hence $\EE[\delta^p]\leq C_p~(d/N)^{p/2}$ for any $p\geq1$; since $d/N<1$, this implies that $\EE[(2\delta+\delta^2)^k]\leq C_k~(d/N)^{k/2}$, and the first estimate in (iii) follows. The second one follows from the first one and the triangle inequality $\Vert\widehat{\sigma}\Vert_2\leq\Vert\sigma\Vert_2+\Vert\widehat{\sigma}-\sigma\Vert_2$.
\end{proof}

The same estimates hold for $(\widehat{n},\widehat{\varsigma})$, which is independent of $(\widehat{m},\widehat{\sigma})$. In the point-mass case we have $b_{\sigma,k}=\Vert\widehat{\sigma}\Vert_2^k$ and $b_{\varsigma,k}=\Vert\widehat{\varsigma}\Vert_2^k$. Hence, by Lemma~\ref{lem:wishart} (iii), the independence of the two samples and (A\ref{ass-1}), there exists some $\texttt{b}<\infty$ depending only on $\Cu$ such that
$$
\EE[\texttt{B}^2]\le \texttt{b}.
$$
Combining Theorem~\ref{theo-main} with the above estimates, we readily check the following corollary.

\begin{cor}\label{cor-gauss-rate}
Assume (A\ref{ass-1}). Then there exists a $C<\infty$ only depending on $(\Cl,\Cu,\cchi)$ such that for any $N>d\geq 1$ we have that
\begin{equation}\label{rate-gauss}
\EE\left[\Da_2\left(\pi^\star,\widehat{\pi}^\star\right)^2\right]\leq\frac{Cd^2}{N}.
\end{equation}
\end{cor}
\begin{proof}
Taking expectations in (\ref{main-bound}) and using the Cauchy-Schwarz inequality
$$
\EE[\texttt{B}~\Delta]\leq\EE[\texttt{B}^2]^{1/2}~\EE[\Delta^2]^{1/2},
$$
it suffices to note that, by Lemma~\ref{lem:wishart} and (A\ref{ass-1}), we have
$$
\EE[\Delta_{m,2}]=\frac{\tr(\sigma)}{N+1}\leq\frac{\Cu~d}{N}
\quad\mbox{\rm and}\quad
\EE[\Delta_{\sigma,2}]\leq\EE[\Delta_{\sigma,2}^2]^{1/2}=\EE[\Delta_{\sigma,4}]^{1/2}\leq C~\frac{d}{N},
$$
and that the same estimates hold for $(\Delta_{n,2},\Delta_{\varsigma,2},\Delta_{\varsigma,4})$. Hence every term in the bracket of (\ref{main-bound}) has an expectation of order $d/N$, which yields (\ref{rate-gauss}). This ends the proof of the corollary.
\end{proof}

\begin{rmk}\label{rmk-one-sided}
If one of the two marginals is known exactly (say $(n,\varsigma)$ is known and only $(m,\sigma)$ is estimated), the same argument yields (\ref{rate-gauss}) with the terms $\Delta_{n,2}$ and $\Delta_{\varsigma,2}$ set to zero. This result can also be extended to dynamic Schr\"odinger bridges (see \cite[Section 2.4]{sinkhorn_rev}), which are widely used in generative models.
\end{rmk}

\subsection{The Monge limit}\label{sec:Monge-limit}

For a single pair of Gaussian marginals, the Schr\"odinger bridge with quadratic cost converges to the optimal transport map as the regularization vanishes. More precisely, we consider the quadratic-cost regime $\beta=I$ and $\tau=tI$ with $t>0$, so that $\cchi=t^{-1}I$; the parameter $t$ is the regularization parameter. By \cite[Corollary~3.13]{adm-24}, we have
\begin{equation}\label{Monge-approx}
\Big\Vert\frac{\kappa}{t}-\sigma^{-1}\sharp\varsigma\Big\Vert_2\vee\Big\Vert\frac{r}{t}-\overline{\varpi}^{1/2}\Big\Vert_2\leq c~t
\quad\mbox{\rm with}\quad
\overline{\varpi}^{-1}:=\varsigma^{1/2}~\sigma~\varsigma^{1/2},
\end{equation}
for some finite constant $c$ depending only on $(\sigma,\varsigma)$. Indeed, $\kappa\cchi=t^{-1}\kappa\to\sigma^{-1}\sharp\varsigma$ by (\ref{Monge-approx}), whilst
$$
\EE\left[\Vert\kappa^{1/2}G\Vert_2^2\right]=\tr(\kappa)\leq d~\Vert\kappa\Vert_2\leq d~t~\left(\Vert\sigma^{-1}\sharp\varsigma\Vert_2+c~t\right)\longrightarrow_{t\rightarrow 0}0 .
$$
In this context we have the limiting Monge map
$$
Y(x)\longrightarrow_{t\rightarrow 0} T(x)=n+\zeta~(x-m)
\quad\mbox{\rm with}\quad
\zeta:=\sigma^{-1}\sharp\varsigma=\sigma^{-1/2}\bigl(\sigma^{1/2}~\varsigma~\sigma^{1/2}\bigr)^{1/2}\sigma^{-1/2}.
$$
The map $T$ is the Monge map (\ref{def-Monge}) pushing $\nu_{m,\sigma}$ forward to $\nu_{n,\varsigma}$, optimal for the quadratic cost (\ref{eq:Monge}); it satisfies the Riccati identity $\zeta~\sigma~\zeta=\varsigma$. Note that, under (A\ref{ass-1}), we have
$$
\Vert\zeta\Vert_2^2\leq\Vert\varsigma\Vert_2~\Vert\sigma^{-1}\Vert_2\leq \Cu/\Cl .
$$

Let $(\widehat{m},\widehat{\sigma})$ and $(\widehat{n},\widehat{\varsigma})$ be estimates of the Gaussian parameters $(m,\sigma)$ and $(n,\varsigma)$. The plug-in estimate of $\zeta$ requires the inversion of $\widehat\sigma$, which may be ill conditioned. To avoid this problem, we inflate the covariance in the outer factors only: for a fixed inflation parameter $\epsilon>0$ we set
$$
\widehat{\sigma}_{\epsilon}:=\widehat{\sigma}+\epsilon I
\quad\mbox{\rm and}\quad
\widehat{\zeta}_{\epsilon}:=\widehat{\sigma}_{\epsilon}^{-1/2}\bigl(\widehat{\sigma}^{1/2}~\widehat{\varsigma}~\widehat{\sigma}^{1/2}\bigr)^{1/2}\widehat{\sigma}_{\epsilon}^{-1/2}.
$$
The regularized empirical Monge map is obtained by replacing $(m,n,\zeta)$ by $(\widehat m,\widehat n,\widehat\zeta_\epsilon)$; that is, we set
\begin{equation}\label{empirical-Monge}
\widehat{T}_{\epsilon}(x):=\widehat{n}+\widehat{\zeta}_{\epsilon}~(x-\widehat{m}).
\end{equation}
The error of $\widehat{T}_{\epsilon}$ is measured in $L^2(\nu_{m,\sigma})$ and averaged over the samples.

\begin{theo}\label{theo-Monge}
Assume (A\ref{ass-1}), and let $(\widehat{m},\widehat{\sigma})$ and $(\widehat{n},\widehat{\varsigma})$ be the sample moments (\ref{sample}) built from two independent samples. Then there exist constants $C,c_1<\infty$ and $c_0>0$ only depending on $(\Cl,\Cu)$ such that the following assertions hold.
\begin{enumerate}
\item[(i)] For every $\epsilon> 0$ and $N>d\geq 1$ we have that
\begin{equation}\label{ineq:Monge2}
\EE\left[\int\nu_{m,\sigma}(dx)~\bigl\Vert\widehat{T}_{\epsilon}(x)-T(x)\bigr\Vert_2^2\right]
\le Cd^2\left(\frac{1+\epsilon^{-1}}{N}+\epsilon^2\right).
\end{equation}
\item[(ii)] If in addition $N\geq c_1 d$, then for every $\epsilon>0$ we have that
\begin{equation}\label{ineq:Monge3}
\EE\left[\int\nu_{m,\sigma}(dx)~\bigl\Vert\widehat{T}_{\epsilon}(x)-T(x)\bigr\Vert_2^2\right]
\le C\left(\frac{d^2}{N}+d~\epsilon^2+d~(1+\epsilon^{-1})~e^{-c_0N}\right).
\end{equation}
\end{enumerate}
\end{theo}

\begin{proof}
Throughout the proof, $C$ denotes a finite constant depending only on $(\Cl,\Cu)$, whose value may change from line to line. We set
$$
X:=\int\nu_{m,\sigma}(dx)~\Vert\widehat{T}_{\epsilon}(x)-T(x)\Vert_2^2 .
$$
We have
$$
\widehat{T}_{\epsilon}(x)-T(x)=(\widehat{n}-n)+\widehat{\zeta}_{\epsilon}~(m-\widehat{m})+(\widehat{\zeta}_{\epsilon}-\zeta)~(x-m).
$$
Taking squared norms, using $(a+b+c)^2\leq3(a^2+b^2+c^2)$, integrating w.r.t.~$\nu_{m,\sigma}(dx)$ and recalling that $\tr(\sigma)\leq\Cu~d$, we find that
\begin{equation}\label{Monge-decomp}
X\leq 3~\Vert\widehat{n}-n\Vert_2^2+3~\Vert\widehat{\zeta}_{\epsilon}\Vert_2^2~\Vert\widehat{m}-m\Vert_2^2+3~\Cu~d~\Vert\widehat{\zeta}_{\epsilon}-\zeta\Vert_2^2.
\end{equation}

\emph{Preliminary bound.} We set
$$
\widehat{A}_{\epsilon}:=\widehat{\sigma}_{\epsilon}^{-1/2},\qquad
\widehat{B}:=\bigl(\widehat{\sigma}^{1/2}~\widehat{\varsigma}~\widehat{\sigma}^{1/2}\bigr)^{1/2},\qquad
A:=\sigma^{-1/2}
\quad\mbox{\rm and}\quad
B:=\bigl(\sigma^{1/2}~\varsigma~\sigma^{1/2}\bigr)^{1/2},
$$
so that $\widehat{\zeta}_{\epsilon}=\widehat{A}_{\epsilon}\widehat{B}\widehat{A}_{\epsilon}$ and $\zeta=ABA$. Since $\widehat{B}^2\leq\Vert\widehat{\varsigma}\Vert_2~\widehat{\sigma}$ and the square root is operator monotone (see for instance \cite[Chapter V]{bhatia-book}), we have $\widehat{B}\leq\Vert\widehat{\varsigma}\Vert_2^{1/2}~\widehat{\sigma}^{1/2}$. Recalling that $\widehat{\sigma}$ and $\widehat{\sigma}_{\epsilon}$ commute, this implies that
$$
0\leq\widehat{\zeta}_{\epsilon}\leq\Vert\widehat{\varsigma}\Vert_2^{1/2}~\widehat{\sigma}^{1/2}~\widehat{\sigma}_{\epsilon}^{-1}
\Longrightarrow
\Vert\widehat{\zeta}_{\epsilon}\Vert_2\leq\Vert\widehat{\varsigma}\Vert_2^{1/2}~\sup_{\lambda\geq0}\frac{\sqrt{\lambda}}{\lambda+\epsilon}=\frac{\Vert\widehat{\varsigma}\Vert_2^{1/2}}{2\sqrt{\epsilon}}.
$$
Using the estimate
$$
\Vert\widehat{\zeta}_{\epsilon}-\zeta\Vert_2^2\leq2~\Vert\widehat{\zeta}_{\epsilon}\Vert_2^2+2~\Vert\zeta\Vert_2^2\leq2~\Vert\widehat{\zeta}_{\epsilon}\Vert_2^2+2~\Cu/\Cl
$$
in (\ref{Monge-decomp}) we obtain
\begin{equation}\label{Monge-apriori}
X\leq 3~\Vert\widehat{n}-n\Vert_2^2+\frac{3}{4\epsilon}~\Vert\widehat{\varsigma}\Vert_2~\Vert\widehat{m}-m\Vert_2^2+C~d~\left(1+\frac{\Vert\widehat{\varsigma}\Vert_2}{\epsilon}\right).
\end{equation}
Since $\widehat{\varsigma}$ is independent of $\widehat{m}$, Lemma~\ref{lem:wishart} and (A\ref{ass-1}) yield
\begin{equation}\label{Monge-apriori-2}
\EE[X]\leq C~d~(1+\epsilon^{-1})
\quad\mbox{\rm and}\quad
\EE[X^2]^{1/2}\leq C~d~(1+\epsilon^{-1}).
\end{equation}

\emph{Refined bound for $N\gtrsim d$.} Consider the event
$$
\Omega_N:=\left\{\Vert\widehat{\sigma}-\sigma\Vert_2\leq\Cl/2\right\}\cap\left\{\Vert\widehat{\varsigma}-\varsigma\Vert_2\leq\Cl/2\right\}.
$$
On $\Omega_N$, by Weyl's inequality and (A\ref{ass-1}), we have
$$
\lambda_{\min}(\widehat{\sigma})\geq\Cl/2
\quad\mbox{\rm and}\quad
\Vert\widehat{\sigma}\Vert_2\vee\Vert\widehat{\varsigma}\Vert_2\leq 3\Cu/2.
$$
Consequently, uniformly w.r.t.~the parameter $\epsilon>0$, we have
$$
\Vert\widehat{\sigma}_{\epsilon}^{-1}\Vert_2\leq\frac{2}{\Cl},\qquad
\Vert\widehat{A}_{\epsilon}\Vert_2\leq\left(\frac{2}{\Cl}\right)^{1/2}
\quad\mbox{\rm and}\quad
\Vert\widehat{B}\Vert_2\leq\left(\Vert\widehat{\sigma}\Vert_2~\Vert\widehat{\varsigma}\Vert_2\right)^{1/2}\leq\frac{3\Cu}{2},
$$
as well as $\Vert A\Vert_2\leq\Cl^{-1/2}$ and $\Vert B\Vert_2\leq\Cu$. By (\ref{square-root-key-estimate}) applied to $\widehat{\sigma}_{\epsilon}^{-1}$ and $\sigma^{-1}$, and recalling that $\lambda_{\min}(\sigma^{-1})\geq\Cu^{-1}$, we have
$$
\Vert\widehat{A}_{\epsilon}-A\Vert_2\leq\sqrt{\Cu}~\Vert\widehat{\sigma}_{\epsilon}^{-1}-\sigma^{-1}\Vert_2
=\sqrt{\Cu}~\Vert\widehat{\sigma}_{\epsilon}^{-1}(\sigma-\widehat{\sigma}_{\epsilon})\sigma^{-1}\Vert_2
\leq\frac{2\sqrt{\Cu}}{\Cl^2}\left(\Vert\widehat{\sigma}-\sigma\Vert_2+\epsilon\right).
$$
Next, we apply (\ref{square-root-key-estimate}) to $\widehat{B}^2$ and $B^2$, using $\lambda_{\min}(B^2)=\lambda_{\min}(\sigma^{1/2}\varsigma\sigma^{1/2})\geq\Cl^2$, the decomposition
\begin{equation}\label{ineq:B2}
\widehat{B}^2-B^2=(\widehat{\sigma}^{1/2}-\sigma^{1/2})~\widehat{\varsigma}~\widehat{\sigma}^{1/2}+\sigma^{1/2}~(\widehat{\varsigma}-\varsigma)~\widehat{\sigma}^{1/2}+\sigma^{1/2}~\varsigma~(\widehat{\sigma}^{1/2}-\sigma^{1/2}),
\end{equation}
and the estimate $\Vert\widehat{\sigma}^{1/2}-\sigma^{1/2}\Vert_2\leq\Cl^{-1/2}~\Vert\widehat{\sigma}-\sigma\Vert_2$. On $\Omega_N$ we find that
$$
\Vert\widehat{B}-B\Vert_2\leq\frac{1}{\Cl}~\Vert\widehat{B}^2-B^2\Vert_2\leq C~\left(\Vert\widehat{\sigma}-\sigma\Vert_2+\Vert\widehat{\varsigma}-\varsigma\Vert_2\right).
$$
Plugging these estimates into the decomposition
\begin{equation}\label{ineq:geom-mean-diff}
\widehat{\zeta}_{\epsilon}-\zeta
=(\widehat{A}_{\epsilon}-A)~\widehat{B}~\widehat{A}_{\epsilon}
+A~(\widehat{B}-B)~\widehat{A}_{\epsilon}
+A~B~(\widehat{A}_{\epsilon}-A),
\end{equation}
we conclude that on $\Omega_N$ we have
$$
\Vert\widehat{\zeta}_{\epsilon}\Vert_2\leq C
\quad\mbox{\rm and}\quad
\Vert\widehat{\zeta}_{\epsilon}-\zeta\Vert_2\leq C~\left(\Vert\widehat{\sigma}-\sigma\Vert_2+\Vert\widehat{\varsigma}-\varsigma\Vert_2+\epsilon\right).
$$
Returning to (\ref{Monge-decomp}) and using Lemma~\ref{lem:wishart}, we obtain
\begin{equation}\label{Monge-good}
\begin{split}
\EE\left[X~1_{\Omega_N}\right]&\leq C\left(\EE\left[\Vert\widehat{n}-n\Vert_2^2\right]+\EE\left[\Vert\widehat{m}-m\Vert_2^2\right]+d\left(\EE\left[\Vert\widehat{\sigma}-\sigma\Vert_2^2\right]+\EE\left[\Vert\widehat{\varsigma}-\varsigma\Vert_2^2\right]+\epsilon^2\right)\right)\\
&\leq C\left(\frac{d^2}{N}+d~\epsilon^2\right).
\end{split}
\end{equation}

On the bad event, we set $c_1:=144~\Cu^2/\Cl^2$ and we assume that $N\geq c_1d$, so that $\sqrt{d/N}\leq\Cl/(12\Cu)$. Applying Lemma~\ref{lem:wishart} (ii) with $\eta:=\Cl/(12\Cu)$ to both samples, and noting that $3\Cu(\sqrt{d/N}+\eta)\leq\Cl/2$, we find that
$$
\Prob\left(\Omega_N^c\right)\leq 4\exp\left(-\frac{N~\Cl^2}{288~\Cu^2}\right).
$$
By the Cauchy-Schwarz inequality and (\ref{Monge-apriori-2}) this implies that
$$
\EE\left[X~1_{\Omega_N^c}\right]\leq\EE[X^2]^{1/2}~\Prob(\Omega_N^c)^{1/2}\leq C~d~(1+\epsilon^{-1})~e^{-c_0N}
\quad\mbox{\rm with}\quad c_0:=\frac{\Cl^2}{576~\Cu^2}.
$$
Combined with (\ref{Monge-good}), this proves (\ref{ineq:Monge3}).

It remains to deduce \eqref{ineq:Monge2}. If $N\geq c_1d$, the estimate (\ref{ineq:Monge3}) and the inequality $e^{-c_0N}\leq (e~c_0N)^{-1}$ yield
$$
\EE[X]\leq C\left(\frac{d^2}{N}+d~\epsilon^2+\frac{d~(1+\epsilon^{-1})}{N}\right),
$$
which implies (\ref{ineq:Monge2}). If $d<N<c_1d$, then by (\ref{Monge-apriori-2}) we have
$$
\EE[X]\leq C~d~(1+\epsilon^{-1})\leq C~c_1~\frac{d^2}{N}~(1+\epsilon^{-1}).
$$
This ends the proof of the theorem.
\end{proof}

\begin{rmk}
We use $\epsilon$-inflation here to avoid ill-conditioned matrix estimates, which is also known as Tikhonov regularization. There is no such ill-conditioning problem in Gaussian Schr\"odinger bridges and Gaussian mixture bridges, thanks to Proposition~\ref{prop:ricc-alt}. Last, but not least, Theorem~\ref{theo-Monge} (ii) shows that, as soon as $N\gtrsim d$, the inflation only enters through the bias term $d~\epsilon^2$ and an exponentially small term: with high probability $\widehat{\sigma}$ is well conditioned and $\widehat{\zeta}_{\epsilon}$ is a stable estimate of $\zeta$, uniformly in $\epsilon$. The factor $\epsilon^{-1}$ in (\ref{ineq:Monge2}) is thus only relevant in the regime $N\asymp d$.
\end{rmk}

\section{Numerical Illustrations}\label{sec:num}

We illustrate the theoretical results on three families of experiments. Section~\ref{sec:num-mixture} illustrates the mixture bridge on a two-component, two-dimensional toy model, including the case where the mixture components are themselves estimated from unlabeled samples by an EM algorithm. Section~\ref{sec:num-gauss} considers a single pair of Gaussian marginals in dimension $d=20$, and the Monge map; these experiments illustrate the convergence rates of the empirical bridge (Corollary~\ref{cor-gauss-rate}) and of the regularized empirical Monge map (Theorem~\ref{theo-Monge}), and probe the effect of the inflation parameter $\epsilon$.

\subsection{Gaussian mixtures}\label{sec:num-mixture}

We illustrate the mixture bridge of Section~\ref{sec:product} on a two-component, two-dimensional model with common component covariances. The source and target mixtures are
$$
p=\sum_{u=1}^2 w_u~\nu_{m_u,\sigma}
\quad\mbox{\rm and}\quad
q=\sum_{v=1}^2 w_v^{\prime}~\nu_{n_v,\varsigma},
$$
with weights $w=(0.6,0.4)$ and $w^{\prime}=(0.5,0.5)$, means $m_1=(0,0)$, $m_2=(3.5,3)$, $n_1=(4.5,0)$, $n_2=(0.5,4.5)$, and common covariance matrices
$$
\sigma=\begin{pmatrix}1&0.3\\0.3&1\end{pmatrix}
\quad\mbox{\rm and}\quad
\varsigma=\begin{pmatrix}0.9&0.2\\0.2&0.9\end{pmatrix}.
$$
The reference transition is given by $\tau=tI$ and $\beta=I$ with $t=0.5$, and the label bridge is independent, $J(u,\point)=\jmath$, so that $\pi_{u,v}=w_u w_v^{\prime}$. Figure~\ref{fig:matching} illustrates the resulting bridge coupling: for each component pair $(u,v)$, the bridge transition $\BB_{u,v}$ is Gaussian with the Riccati closed form, and we draw $800$ pairs $(X_i,Y_i)$ from the mixture coupling $\pi=\sum_{u,v}w_u w_v^{\prime}~\BB_{u,v}$. Panel (a) displays the source samples $X_i$ together with the contours of $p$, colored by component; panel (b) displays the transported samples $Y_i$ (triangles) together with the contours of $q$, keeping the color of their source component (in Figure~\ref{fig:matching} we use $\mu$ for the source means and $\nu$ for the target means). The transported samples concentrate along the target contours, showing that the bridge pushes $p$ forward to $q$; the mixing of colors within each target component reflects the independent label bridge $\pi_{u,v}=w_u w_v^{\prime}$.

Then both mixtures are fitted from $N$ i.i.d.~samples by the EM algorithm under the equal-covariance Gaussian mixture model (the \texttt{EEE} model of \texttt{mclust}); the bridges are built from the fitted parameters, and the labels are aligned post hoc with the true components. As a benchmark, we consider a supervised oracle that observes the true component labels alongside the samples, and estimates the mixture parameters by the within-component empirical quantities (\ref{sample}). We use the right-hand side of (\ref{ext}), which in this finite-label, independent-bridge setting reads
\begin{multline*}
\mathrm{UB}\big(\widehat{\PP}\times\widehat{\MM},\PP\times\MM\big):=(1+a^2)^{1/2}\Big(\sum_{u}\widehat{w}_u~\Da_2^2\big(\nu_{\widehat{m}_u,\widehat{\sigma}},\nu_{m_u,\sigma}\big)\Big)^{1/2}
+\sqrt{2}~(1+a^2)^{1/2}~\Da_2\big(\widehat{\jmath},\jmath\big)\\
+(1+a^2)~\Da_2\big(\widehat{\imath},\imath\big)
+\sqrt{2}~\Big(\sum_{u,v}\widehat{\pi}_{u,v}~\EE_{x\sim\nu_{\widehat{m}_u,\widehat{\sigma}}}\Da_2^2\big(\widehat{\BB}_{u,v}(x,\point),\BB_{u,v}(x,\point)\big)\Big)^{1/2}.
\end{multline*}
We repeat each experiment $50$ times for $N\in\{100,200,500,1000,2000,5000\}$.

Figure~\ref{fig:mix} displays $\EE[\mathrm{UB}^2]$ against $N$ on the log-log scale, for both the EM estimator and the supervised oracle. Two observations are in order. First, the two curves are nearly indistinguishable across the whole range of $N$: under good separation and reasonable initialization, the latent labels incur essentially no loss, and the error is driven by parameter estimation rather than label ambiguity. Second, the fitted slopes ($-0.55$ for EM and $-0.58$ for the supervised oracle) are consistent with the rate stated in Corollary~\ref{cor:em}. Indeed, our bound decomposes into a mixing-weight term $O_p(N^{-1/2})$ and location and scale terms $O_p(N^{-1})$. A sum of an $N^{-1/2}$ term and an $N^{-1}$ term exhibits a finite-sample log-log slope strictly between $-1$ and $-1/2$, approaching $-1/2$ as $N$ grows; the observed slopes fall precisely in this regime.

\begin{figure}[t]
  \centering
  \includegraphics[width=1\textwidth,height=7cm]{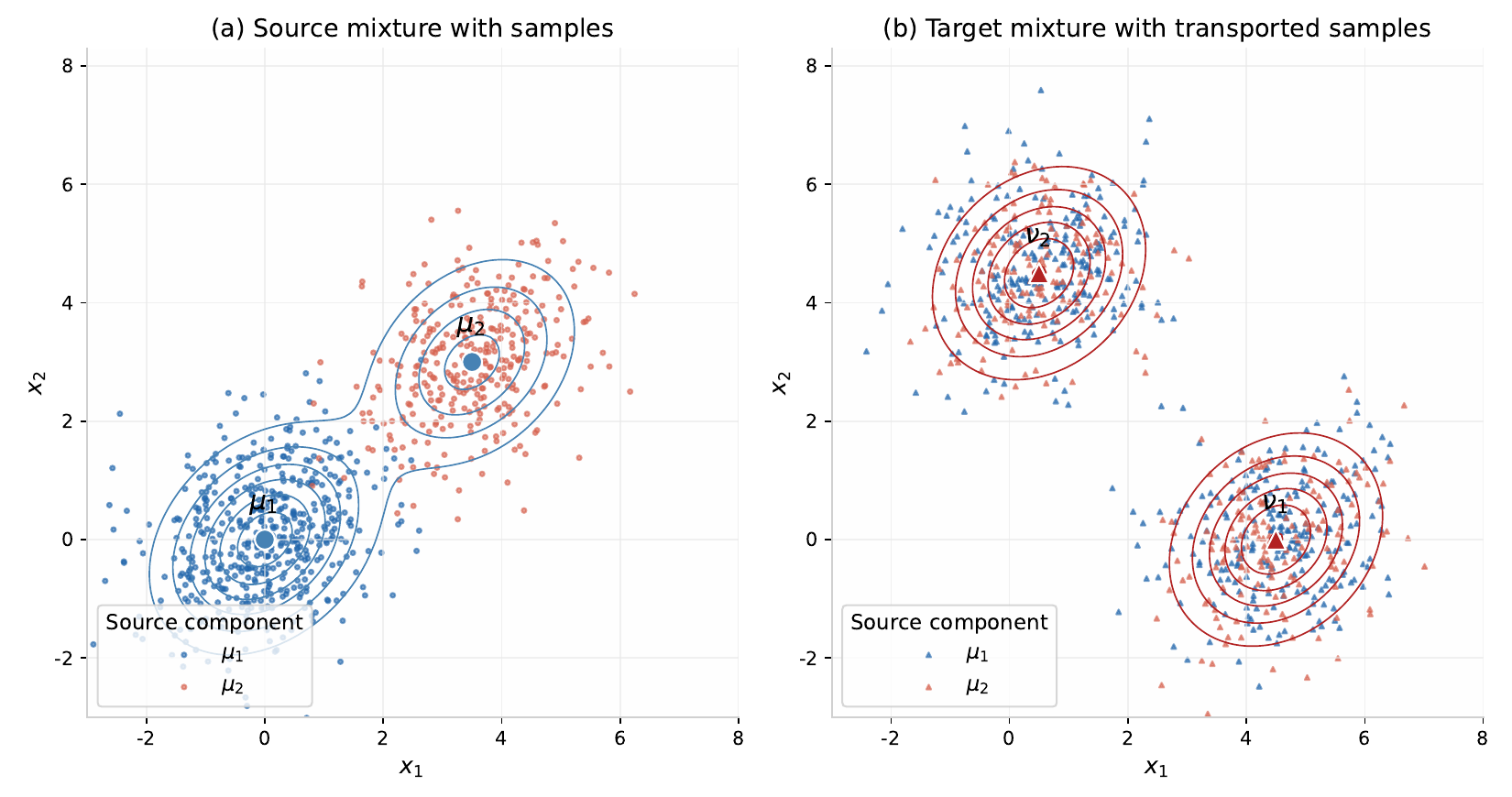}
  \caption{Left: the source mixture $p$ with i.i.d.~samples, colored by component. Right: the target mixture $q$ with the same samples transported through the bridge coupling (triangles), keeping the color of their source component.}\label{fig:matching}
\end{figure}

\begin{figure}[t]
  \centering
  \includegraphics[width=0.8\textwidth,height=6cm]{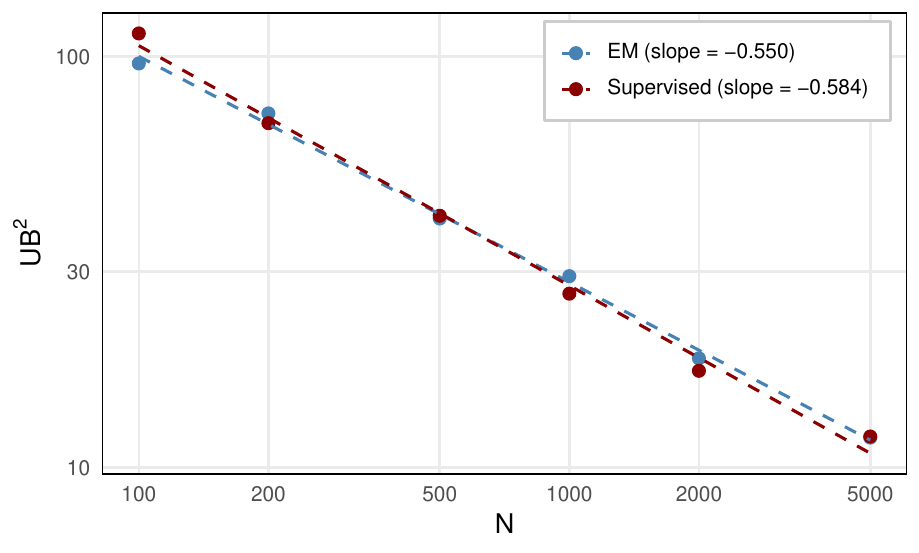}
  \caption{Convergence rates using EM for the Gaussian mixture bridge, measured by $\EE[\mathrm{UB}^2]$ against $N$ on log-log scales.}\label{fig:mix}
\end{figure}

\subsection{Single Gaussian marginals and Monge map}\label{sec:num-gauss}

We illustrate the finite-sample results of Section~\ref{sec:empirical} on a single pair of Gaussian marginals in dimension $d=20$. The source and target parameters are $m=n=0$ and
$$
\sigma_{ij}=\rho^{|i-j|}\quad\mbox{\rm with}\quad\rho=0.3,
\qquad
\varsigma=I_d+0.2~\mathbf{1}\mathbf{1}^{\prime}.
$$
The reference transition is specified by
$$
\tau=I_d+0.1~\mathbf{1}\mathbf{1}^{\prime}
\quad\mbox{\rm and}\quad
\beta=U~\mathrm{diag}(s_1,\dots,s_d)~V^{\prime},
$$
with singular values $s_k$ equally spaced in $[0.5,1.5]$ and $U,V$ two independent Haar orthogonal matrices. For the Monge map, we use the same parameters $(m,n,\sigma,\varsigma)$. The estimates are given by (\ref{sample}) and (\ref{empirical-Monge}), and the errors are measured by $\Da_2(\pi^\star,\widehat\pi^\star)^2$ for the Gaussian bridge, and by
$$
\mathcal{T}_2^2:=\int\nu_{m,\sigma}(dx)~\bigl\Vert\widehat T_\epsilon(x)-T(x)\bigr\Vert_2^2
$$
for the Monge map. All expectations below are estimated over $150$ independent replications.

\begin{figure}[t]
  \centering
  \includegraphics[width=0.8\textwidth,height=8cm]{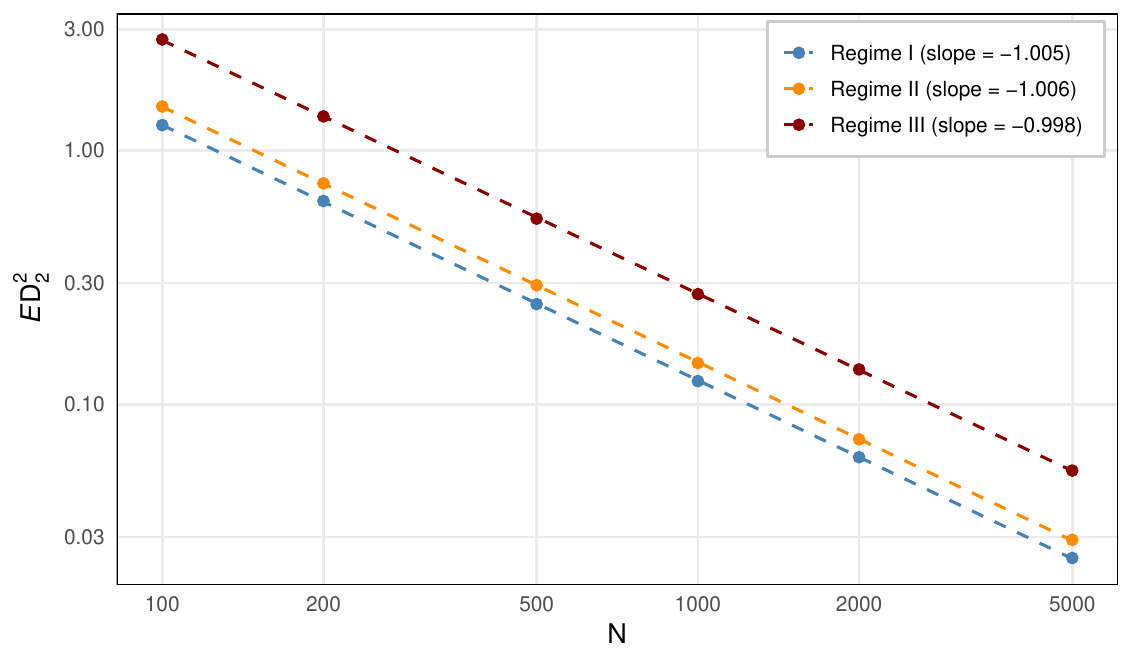}
  \caption{Convergence rates of the empirical Gaussian Schr\"odinger bridge, measured by the squared $2$-Wasserstein distance against $N$ on log-log scales. Regimes I and II are one-sided estimates and Regime III is for two-sided estimates.}\label{fig:gauss}
\end{figure}
\begin{figure}[t]
  \centering
  \includegraphics[width=0.8\textwidth,height=6cm]{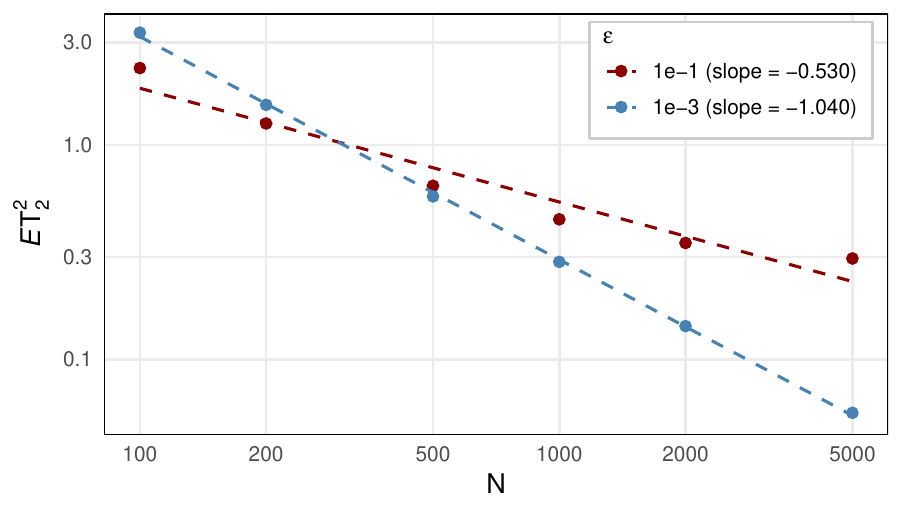}
  \caption{Convergence rates of the regularized empirical Monge map, measured by $\mathcal{T}_2^2$ against $N$ on log-log scales.}\label{fig:Monge}
\end{figure}

Figures~\ref{fig:gauss} and \ref{fig:Monge} both display the errors against $N\in\{100,200,500,1000,2000,5000\}$. In Figure~\ref{fig:gauss}, we test three regimes: in Regimes I and II we estimate one marginal Gaussian (the start and the end marginal, respectively) with the other side being known, and in Regime III we estimate both sides. All three regimes show a convergence rate of order $N^{-1}$, which illustrates Corollary~\ref{cor-gauss-rate} and Remark~\ref{rmk-one-sided}. In Figure~\ref{fig:Monge}, we test two inflation parameters, $\epsilon=0.001$ and $\epsilon=0.1$. With the small inflation parameter $\epsilon=0.001$, the convergence rate is of order $N^{-1}$ with a moderate constant, which is consistent with Theorem~\ref{theo-Monge} (ii): once $N$ is a sufficiently large multiple of $d$, the factor $\epsilon^{-1}$ only multiplies an exponentially small term. When $\epsilon$ grows to $0.1$, Figure~\ref{fig:Monge} shows that the convergence of the regularized empirical Monge map is clearly affected by the bias term $d~\epsilon^2$ in (\ref{ineq:Monge3}).

\appendix

\section{Proofs of Technical Results}

\subsection{Proof of Lemma~\ref{Lem-Tex}}\label{sec:lem_prf}

We have the triangle inequality
$$
\Da_2\left(\mu\times K,\overline{\mu}\times\overline{K}\right)\leq \Da_2\left(\mu\times K,\mu\times \overline{K}\right)+
\Da_2\left(\mu\times \overline{K},\overline{\mu}\times \overline{K}\right).
$$
Let $Q_{x}$ be a $2$-Wasserstein optimal coupling between $K(x,\point)$ and $\overline{K}(x,\point)$, so that
$$
\Da_2\left(K(x,\point),\overline{K}(x,\point)\right)^2=\int Q_{x}(d(y,\overline{y}))~\rho_{\YY}(y,\overline{y})^2 .
$$
Choosing the coupling
$$
\mu(dx)~\delta_{x}(d\overline{x})~Q_{x}(d(y,\overline{y}))
$$
between $\mu(dx)K(x,dy)$ and $\mu(d\overline{x})\overline{K}(\overline{x},d\overline{y})$, we check that
\begin{eqnarray*}
\Da_2\left(\mu\times K,\mu\times \overline{K}\right)^2&\leq &\int \mu(dx)~\delta_{x}(d\overline{x})~Q_{x}(d(y,\overline{y}))~\left(\rho_{\XX}(x,\overline{x})^2+\rho_{\YY}(y,\overline{y})^2\right)\\
&=& \int\mu(dx)~ \Da_2\left(K(x,\point),\overline{K}(x,\point)\right)^2.
\end{eqnarray*}
To estimate $\Da_2\left(\mu\times \overline{K},\overline{\mu}\times \overline{K}\right)$, consider some $2$-Wasserstein optimal coupling $S((x,\overline{x}),d(y,\overline{y}))$ between $\overline{K}(x,dy)$ and $\overline{K}(\overline{x},d\overline{y})$, so that, by (\ref{est-hyp}),
$$
\Da_2\left( \overline{K}(x,\point), \overline{K}(\overline{x},\point)\right)^2=\int
S((x,\overline{x}),d(y,\overline{y}))~\rho_{\YY}(y,\overline{y})^2
\leq a^2~\rho_{\XX}(x,\overline{x})^2
$$
for $\gamma$-almost every $(x,\overline{x})$. Finally, with the coupling $\gamma\in\Ca(\mu,\overline{\mu})$ of the statement, we set
$$
\Gamma(d[(x,y),(\overline{x},\overline{y})]):=\gamma(d(x,\overline{x}))~S((x,\overline{x}),d(y,\overline{y})).
$$
Since $\Gamma$ is a coupling between $\mu\times \overline{K}$ and $\overline{\mu}\times \overline{K}$, we have
\begin{eqnarray*}
\Da_2\left(\mu\times \overline{K},\overline{\mu}\times \overline{K}\right)^2&\leq&
\int \gamma(d(x,\overline{x}))~S((x,\overline{x}),d(y,\overline{y}))~\left(\rho_{\XX}(x,\overline{x})^2+\rho_{\YY}(y,\overline{y})^2\right)\\
&\leq&(1+a^2)~\int \gamma(d(x,\overline{x}))~\rho_{\XX}(x,\overline{x})^2.
\end{eqnarray*}
Combining the two estimates yields (\ref{est-lem-tex}); the last assertion follows by choosing for $\gamma$ an optimal coupling. This ends the proof of the lemma.
\hfill$\square$

\subsection{Proof of Lemma~\ref{lem:h-lipschitz}}\label{app:lip-ricatti}

We write $F=g\circ s$, with
$$
s(\psi):=\left(\psi+\frac{I}{4}\right)^{1/2}
\quad\mbox{\rm and}\quad
g(A):=\left(A+\frac{I}{2}\right)^{-1}.
$$
Since for any $i\in \{1,2\}$ we have $\lambda_{\min}\left(s(\psi_i)\right)\geq 1/2$, the Ando--van Hemmen inequality (\ref{square-root-key-estimate}) gives
$$
\Vert s(\psi_1)-s(\psi_2)\Vert
\leq
\left[\lambda_{\min}\left(s(\psi_1)\right)+\lambda_{\min}\left(s(\psi_2)\right)\right]^{-1}~\Vert\psi_1-\psi_2\Vert
\leq\Vert\psi_1-\psi_2\Vert .
$$
For $A_1,A_2\geq I/2$, the resolvent identity gives
$$
g(A_1)-g(A_2)=\left(A_1+\frac{I}{2}\right)^{-1}(A_2-A_1)\left(A_2+\frac{I}{2}\right)^{-1},
\quad\mbox{\rm with}\quad
\left\Vert\left(A_i+\frac{I}{2}\right)^{-1}\right\Vert_2\leq1 .
$$
Using the ideal property of unitarily invariant norms
$$
\Vert AXB\Vert\leq\Vert A\Vert_2~\Vert X\Vert~\Vert B\Vert_2,
$$
we obtain $\Vert g(A_1)-g(A_2)\Vert\leq\Vert A_1-A_2\Vert$. Combining the two estimates proves (\ref{h-lipschitz}).
\hfill$\square$

\subsection{Proof of Lemma~\ref{lem:psi}}\label{app:lip}

We use the decomposition
\begin{multline*}
\psi(\widehat{\sigma}_u,\widehat{\varsigma}_v)-\psi(\sigma_u,\varsigma_v)
=(\widehat{\varsigma}_v^{1/2}-\varsigma_v^{1/2})~(\cchi~\widehat{\sigma}_u~\cchi^{\prime})~\widehat{\varsigma}_v^{1/2}\\
+\varsigma_v^{1/2}~(\cchi~\widehat{\sigma}_u~\cchi^{\prime})~(\widehat{\varsigma}_v^{1/2}-\varsigma_v^{1/2})
+\varsigma_v^{1/2}~(\cchi~(\widehat{\sigma}_u-\sigma_u)~\cchi^{\prime})~\varsigma_v^{1/2}.
\end{multline*}
This implies that
\begin{multline*}
\Vert\psi(\widehat{\sigma}_u,\widehat{\varsigma}_v)-\psi(\sigma_u,\varsigma_v)\Vert_2
\leq \Vert \widehat{\varsigma}_v^{1/2}-\varsigma_v^{1/2}\Vert_2~\Vert \cchi\Vert_2^2~\Vert\widehat{\sigma}_u\Vert_2\left(\Vert\varsigma_v\Vert_2^{1/2}+\Vert\widehat{\varsigma}_v\Vert^{1/2}_2\right)\\
+\Vert \cchi\Vert_2^2~\Vert\varsigma_v\Vert_2~\Vert \widehat{\sigma}_u-\sigma_u\Vert_2.
\end{multline*}
On the other hand, by (\ref{square-root-key-estimate}) we have
$$
\Vert \widehat{\varsigma}_v^{1/2}-\varsigma_v^{1/2}\Vert_2\leq \frac{1}{\lambda^{1/2}_{\min}(\varsigma_v)}~\Vert \widehat{\varsigma}_v-\varsigma_v\Vert_2.
$$
Plugging this estimate into the previous one completes the proof.
\hfill$\square$


\begin{thebibliography}{99.}

\bibitem{adm-24}
{\sc Akyildiz},  O. D. ,  {\sc Del Moral,} P. \&  {\sc Miguez}, J.~(2026).  Gaussian entropic optimal transport: Schr\"odinger bridges and the Sinkhorn algorithm.  \emph{Found.  Data Sci.} (to appear).

\bibitem{alberg}
{\sc Albergo}, M.  \&  {\sc Vanden-Eijnden}, E.~(2023). Building normalizing flows with stochastic interpolants.
In \emph{ICLR}.
 
  \bibitem{hemmen}
 {\sc Ando}, T. \& {\sc van Hemmen}, J.L.~(1980). An inequality for trace ideals.  \emph{Commun. Math. Phys.}, 
{\bf 76},  143--148.

 
 \bibitem{anderson}
{\sc Anderson}, T. W.~(2003). \emph{An Introduction to Multivariate Statistical Analysis}, 3rd edition. Wiley: Hoboken.

\bibitem{bhatia-book}
{\sc Bhatia}, R.~(1997). \emph{Matrix Analysis}. Springer: New York.

\bibitem{bathia-2}
{\sc Bhatia},  R., {\sc Jain}, T., \& {\sc Lim}, Y.~(2019).  On the Bures-Wasserstein distance between positive definite matrices.  \emph{Exposit.  Math.},   {\bf 37},  165--191.

\bibitem{bing}
{\sc Bing}, X., {\sc Kong}, D., \& {\sc Li}, B.~(2026). Convergence and Optimality of the EM Algorithm Under Multi-Component Gaussian Mixture Models. \emph{Biometrika}, asag047.

\bibitem{chen}
{\sc Chen}, Y.,  {\sc Georgiou}, T. T. \& {\sc Tannenbaum},  A.~(2019). Optimal Transport for Gaussian Mixture Models.
\emph{IEEE Access}, {\bf 7},  6269--6278.

\bibitem{csis}
{\sc Csiszar},  I.~(1975). I-Divergence Geometry of Probability Distributions and Minimization Problems. 
\emph{Ann. Probab.},  {\bf 3}, 146--158.

\bibitem{cut}
{\sc Cuturi},  M.~(2013).  Sinkhorn distances: Lightspeed computation of optimal transport. In 
\emph{Adv.  Neur.  Inf.  Proc. Sys.},  2292--2300.

\bibitem{bort}
{\sc de Bortoli},  V.,  {\sc Thornton},  J.,  {\sc Heng}, J. \& {\sc Doucet}, A.~(2021)
Diffusion Schr\"odinger bridge with
applications to score-based generative modeling. 
In \emph{Adv.  Neur.  Inf.  Proc. Sys.} 17695--17709.

\bibitem{sinkhorn_rev}
{\sc Del Moral}, P. \& {\sc Jasra},  A.~(2026).  New trends in the stability of Sinkhorn semigroups.  \emph{J.  Theor.  Probab.}, {\bf 39},  article 77.

\bibitem{delon}
{\sc Delon}, J., \& {\sc Desolneux}, A.~(2020). A Wasserstein-type distance in the space of Gaussian mixture models. \emph{SIAM Journal on Imaging Sciences}, {\bf 13}, 936--970.

\bibitem{genev}
{\sc Genevay}, A.,  {\sc Cuturi},  M.  \& {\sc Peyr\'e}, G.~(2018).
Learning generative models with Sinkhorn divergences.  In \emph{AISTATS}
1608--1617.

\bibitem{gold}
{\sc Goldfeld}, Z.,  {\sc Kato}, K. ,  {\sc Rioux},  G.,  \& {\sc Sadhu},  S.~(2024). Limit theorems for entropic optimal transport maps and Sinkhorn divergence.  \emph{Elec. J. Stat.}, {\bf 18}, 980--1041.

\bibitem{koro2}
{\sc Gushchin},  N. ,  {\sc Kholkin},  S.,  {\sc Burnaev,} E.  \& {\sc Korotin},  A.~(2024).
Light and optimal Schr\"odinger bridge matching.  In \emph{ICML}.

\bibitem{har}
{\sc Harchaoui}, Z., {\sc Liu}, L.,  \& {\sc Pal},  S.~(2024).  Asymptotics of discrete Schr\"odinger bridges via chaos decompositions.  \emph{Bernoulli},  {\bf 30},  1945--1970.

\bibitem{kato}
{\sc Kato}, K. ,~(2026). Large deviations for dynamical Schr\"odinger bridge problems.  \emph{J.  Appl.  Probab.},  {\bf 63}, 668--689.

\bibitem{koro1}
{\sc Korotin},  A., {\sc Gushchin},  N. \& {\sc Burnaev,} E. ~(2024).
Light Schr\"odinger bridge.  In \emph{ICLR}.

\bibitem{leonard}
{\sc L\'eonard}, C.~(2014).  A survey of the Schr\"odinger problem and some of its connections with
optimal transport. \emph{Discrete Contin. Dyn. Syst.},  {\bf 34}, 1533--1574.

\bibitem{mix_approx}
{\sc Li}, J. Q.,  \&  {\sc Barron},  A. R.~(1999).
Mixture density estimation.  In \emph{Adv.  Neur.  Inf.  Proc. Sys.} 

\bibitem{maeda}
{\sc Maeda}, I.,  {\sc Yao}, R. \& {\sc Nitanda},  A.~(2025). 
Statistical Analysis of the Sinkhorn Iterations for Two-Sample Schr\"odinger Bridge Estimation.
In \emph{Adv.  Neur.  Inf.  Proc. Sys.} .

\bibitem{pool}
{\sc Pooladian},  A. A. \& {\sc Niles-Weed},  J.~(2025).   Plug-in estimation of Schr\"odinger bridges.
\emph{SIAM J.  Math.  Data Sci.},  {\bf 7},  1315--1336.

\bibitem{rapak}
{\sc Rapakoulias}, G.,  {\sc Pedram},  A. R.,  {\sc Liu},  F. ,  {\sc Zhu},  L. \& {\sc Tsiotras},  P.~(2025).
Go With the Flow: Fast Diffusion for Gaussian mixture models. 
In \emph{Adv.  Neur.  Inf.  Proc. Sys.}.

\bibitem{rapak1}
{\sc Rapakoulias}, G.,  {\sc Pedram},  A. R.,  \& {\sc Tsiotras},  P.~(2025).
Steering large agent populations using mean-field
Schr\"odinger bridges with Gaussian mixture models.  \emph{IEEE Cont.  Sys. Lett.},  {\bf 9},  1760--1765.

\bibitem{stromme}
{\sc Stromme},  A. J.~(2023).
Sampling From a Schr\"odinger Bridge.  In \emph{ICML}.

\bibitem{vershynin}
{\sc Vershynin}, R.~(2012). Introduction to the non-asymptotic analysis of random matrices. In \emph{Compressed Sensing: Theory and Applications}
(Y. Eldar \& G. Kutyniok, eds.), 210--268. Cambridge University Press: Cambridge.

\bibitem{villani}
{\sc Villani}, C.~(2009). \emph{Optimal Transport: Old and New}. Springer: Berlin.

\bibitem{mix}
{\sc Titterington}, M.,  {\sc Smith},  A. F.  M.,  \& {\sc Makov},  U. E.~(1985).   \emph{Statistical Analysis of Finite Mixture Distributions}.  Wiley: Chichester.

 \end{thebibliography}
\end{document}